\documentclass[12pt]{amsart}
\usepackage[headings]{fullpage}
\usepackage{longtable}
\usepackage{rotating}
\usepackage{amssymb}
\theoremstyle{plain}
\newtheorem{theorem}{Theorem}[section]

\newtheorem{prop}[theorem]{Proposition}
\newtheorem{cor}[theorem]{Corollary}
\theoremstyle{remark}

\theoremstyle{definition}
\newtheorem{defn}{Definition}[section]
\newcommand{\PP}{\mathbb P}
\newcommand{\lra}{\longrightarrow}

\newcommand\QQ{{\mathbb Q}}

\numberwithin{table}{section}
\begin{document}
\title
{Classification of double octic Calabi--Yau threefolds} 
 \author{S{\l}awomir Cynk}
 \address{Institute of Mathematics, Jagiellonian University,
{\L}ojasiewicza 6, 30-348 Krak\'ow, Poland 
}
\address{Institute of Mathematics of the
Polish Academy of Sciences, ul. \'Sniadeckich 8, 00-956 Warszawa,
Poland}

 \email{slawomir.cynk@uj.edu.pl}
\author{Beata Kocel--Cynk}
\address{Institute of Mathematics,
         Cracow University of Technology\\ 
         ul.~Warszawska~24,
         31-155~Krak\'ow,
         POLAND}
\email{bkocel@pk.edu.pl}

\thanks{Research partially supported by the National Science Center grant no. 2014/13/B/ST1/00133}
\thanks{This work was partially supported by the grant 346300 for
  IMPAN from the Simons Foundation and the matching 2015-2019 Polish
  MNiSW fund} 
\keywords{Calabi--Yau threefold, double octic, automorphisms}
\subjclass[2010]{Primary: 14J32, Secondary: 14D06}

\maketitle
\begin{abstract}
In the present paper we propose a combinatorial approach to study the
so called double octic Clabi--Yau threefolds. We use this description
to give a complete  classification of double octics with
$h^{1,2}\le1$ and to derive their geometric properties (Kummer surface
fibrations, automorphisms, special elements in families).
\end{abstract}

\section*{Introduction}

Let $\pi:X\lra \PP^{3}$ be a double covering of the projective space
branched along an arrangement of eight planes. If the arrangement
satisfies mild conditions (no six planes intersect, no four contain a
line) then there exists a resolution of singularities of $X$ which is
a projective Calabi--Yau threefold called a double octic
(\cite{CSz}). Double octic Calabi--Yau are very suitable for explicit
computations, their invariants (topological Euler characteristic,
Hodge numbers) can be easily computed (\cite{CvS}). On the other hand
this class is rich enough to provide examples of several interesting
phenomena.

There exist double octic Calabi--Yau threefolds in characteristic 3
non--liftable to characteristic zero (\cite{CvS3}).  Computation of
Picard--Fuchs operators of one--parameter families of double octic
Calabi--Yau exhibited examples with particular properties: an example
without a point of Maximal Unipotent Monodromy or an example with
three points of Maximal Unipotent Monodromy and different instanton
numbers (\cite{CvS2}). 
Double octic Calabi--Yau threefolds have elliptic curve and K3 surface
fibrations, K3 fibrations (K3 fibrations of rigid double octics were
studied in \cite{Bor}).

Double octic Calabi--Yau threefolds are closely related to
desingalarized fiber products of rational elliptic surfaces
(\cite{Sch}), a double octic can be considered  as a fiber--wise
Kummer construction. This relation and its application to modularity
of certain Calabi--Yau threefolds (also non--rigid cf. \cite{HV}) and
existence of correspondences was studied in \cite{MK2}. 

Double octic Calabi--Yau threefolds were studied by C. Meyer
(\cite{Meyer}), he carried out a systematic study of huge number of
examples with integer coefficients in particular he gave
11 examples of rigid double octics and 63 examples of one--parameter
families. Our main task is to made a complete classification of octic
arrangements, contrary to Meyer we did not study explicit
examples. Instead we used combinatorial data  called
``incidence table'' which are independent of the field considered.
Conversely, from an incidence table we can recover equation of an
octic arrangement and verify if two families with identical incidence
tables are projectively equivalent.
From the incidence table we can read singularities of the arrangement,
presentations as a fiber--wise Kummer fibration, permutations of
planes that preserves the incidences and  projective
transformations of projective space that preserve the octic arrangement.
 
We use this combinatorial approach to produce a complete list of examples
with $h^{12}\le1$ and describe their geometry.

\section{Double octics}

Let $D=D_{1}\cup\dots\cup D_{8}$ be a sum of eight planes in
$\PP^{3}$, we shall call $D$ an \emph{octic arrangement} iff
\begin{itemize}
\item the intersection $D_{i_{1}}\cap\dots\cap D_{i_{6}}$ of any six is empty
  ($1\le i_{1}\le i_{2}\le\dots\le i_{6}\le 8$),
\item the intersection $D_{i_{1}}\cap D_{i_{2}}\cap D_{i_{3}}\cap
  D_{i_{4}}$ of any four do not contain a line i.e. it is empty set or one
  point ($1\le i_{1}\le i_{2}\le i_{3}\le i_{4}\le 8$).
\end{itemize}
The surface $D$  is singular at the intersection points of
components, there are five type of singularities of octic arrangements
satisfying the above two restrictions
\begin{itemize}
\item [$l_{2}$] double line,
\item [$l_{3}$] triple line,
\item [$p_{3}$] triple point (not on a triple line),
\item [$p^{0}_{4}$] fourfold point not lying on a triple line,
\item [$p^{1}_{4}$] fourfold point lying on a one triple line,
\item [$p^{0}_{5}$] fivefold point not lying on a triple line,
\item [$p^{1}_{5}$] fivefold point lying on  one triple line,
\item [$p^{2}_{5}$] fivefold point lying on a one triple line.
\end{itemize}
We denote the numbers of multiple lines and points by
$l_{2}, l_{3},p_{3}, p_{4}^{0}, p_{4}^{1}, p_{5}^{0}, p_{5}^{1}, p_{5}^{2}$. 
These numbers are related by the following two relations \cite[Lem.~3.4]{CSz}
\begin{eqnarray*}
  p_{3}+4p_{4}^{0}+3p_{4}^{1}+10p_{5}^{0}+9p_{5}^{1}+8p_{5}^{2}+l_{3}&=&56\\
  p_{4}^{1}+2p_{5}^{1}+4p_{5}^{2}&=&5l_{3}.
\end{eqnarray*}

The double cover of $\PP^{3}$ branched along $D$ admits a resolution of
singularities $X$ that is a smooth Calabi--Yau threefold, the
topological Euler characteristic of $X$ is given by (\cite[Thm.~3.5]{CSz})
\[e(X)=40+4p_{4}^{0}+3p_{4}^{1}+16p_{5}^{0}+18p_{5}^{1}+20p_{5}^{2}+l_{3}.\]

The Kuranishi space (universal deformation) of a double octic $X$ is
given by the space of equisingular deformations (i.e. family of octic
arrangements preserving the types of singularities) modulo trivial
deformations (induced by projective automorphisms of
$\PP^{3}$). Dimension of the Kuranishi space equals the Hodge number
$h^{1,2}(X)$ and can be computed with computer algebra system via
equisingular ideal (cf. \cite{CvS}).  

C.~Meyer in \cite{Meyer} carried out an extensive computer search for
double octic Calabi--Yau threefolds. His method was to study arrangements with
small integer coefficients, compute for them numbers of singularities
of various types ($l_{2}, l_{3},p_{3}, p_{4}^{0}, p_{4}^{1},
p_{5}^{0}, p_{5}^{1}, p_{5}^{2}$) and the Hodge numbers $h^{1,1},
h^{1,2}$. As those numerical invariants do not classify arrangements
he also classified the types of all subarrengements of six
planes. 

Among 450 types of octic arrangements listed in \cite[App.~A]{Meyer}
there are 11 producing rigid (admitting no deformations, equivalently
with the Hodge number $h^{1,2}=0$) Calabi--Yau threefolds and
63 with $h^{1,2}=1$. 
For the 74 arrangement  Meyer gave sample equations of the eight
planes, we shall use equations and the numbering of the arrangements
from \cite{Meyer} (we only corrected equation of Arr. No. 35 and
change the parametrizations of Arr. No. 275 and 276).  

\section{Incidence table}

Classification of double octic Calabi--Yau threefolds is a delicate
matter, Meyer used two measures to classify octic arrangements: the numerical
invariants (number of singular points of various types and the Hodge
numbers) and the ordered list of types of subarrangements of six
planes. However two completely different octic arrangements
(f.i. rigid arrangements no. 32 and 69 in the Meyer list \cite{Meyer}) can give
birational Calabi--Yau threefolds. On the other hand the numerical
invariants do not determine the double octic. To avoid this
difficulties we shall introduce the following definitions
\begin{defn}
  Arrangements of eight planes $D_{1}=P_{1}^{1}\cup\dots\cup P_{8}^{1}$
  and $D_{2}=P_{1}^{2}\cup\dots\cup P_{8}^{2}$ are called
  combinatorially equivalent iff there is a permutation
  $\sigma\in S_{8}$ such that for each indices $1\le
  i_{1}<i_{2}<\dots<i_{k}\le8$ the intersections 
  \[P_{i_{1}}^{1}\cap\dots\cap P_{i_{k}}^{1} \qquad \text{ and } \qquad
  P_{\sigma(i_{1})}^{2}\cap\dots\cap P_{\sigma(i_{k})}^{2}\] have the same dimension
  and (provided non--empty) give the same singularity type in $D_{1}$
  and $D_{2}$.  
\end{defn}
\begin{defn}
The \emph{incidence table} of an octic arrangement
$D=P_{1}\cup\dots\cup P_{8}$ is the the sorted list of all quadruples $1\le
i_{1}<i_{2}<i_{3}<i_{4}\le8$ such that the planes    
$P_{i_{1}},\dots P_{i_{4}}$ intersects.

The\emph{minimal incidence table} of an octic arrangement is the
minimum of incidences tables over all permutations of the planes. 
\end{defn}
The minimal incidence table uniquely
determines the combinatorial equivalence type of an arrangement. 
\begin{prop}
  Arrangements of eight planes $D_{1}=P_{1}^{1}\cup\dots\cup P_{8}^{1}$
  and $D_{2}=P_{2}^{1}\cup\dots\cup P_{8}^{2}$ are   combinatorially
  equivalent iff they have equal minimal incidence tables.
\end{prop}
\begin{proof}
Obviously equivalent arrangements have equal minimal incidence
tables, to prove the opposite implication we can assume, that (after
reordering  one of them) the arrangement in question have equal
incidence tables.
 
Observe first,  that a set of planes in $\mathbb P^{3}$ intersect iff
any four of them intersect. Now, let $1\le i<j<k\le8$ be any triple of
indices, that the planes $P_{i}^{1},P_{j}^{1},P_{k}^{1}$ intersect
along a line iff for each $l\in\{1,\dots,8\}\setminus\{i,j,k\}$ the
planes $P_{i}^{1},P_{j}^{1},P_{k}^{1},P_{l}^{1}$ intersects. As the
same holds true for planes $P_{i}^{2}$, we get triple lines, fourfold
and fivefold points in both arrangements given by the corresponding
planes. As the type of a singularity is determined by the number of
planes and triple lines through a point, both arrangements are
combinatorially equivalent.
\end{proof}

An octic arrangement can be described by a $8\times 4$ matrix of
coefficients of linear forms defining the planes, four planes of the
arrangements intersect iff the corresponding $4\times4$ minor is
zero. Consequently the incidence table can be computed from the
coefficient matrix by finding vanishing maximal minors. 

For instance for the Arr. 1. given by
\[xyzt(x+y)(y+z)(z+t)(t+x)=0\]
the matrix is given by 
\[\left[
\begin{array}{cccc}
  1&0&0&0\\
  0&1&0&0\\
  0&0&1&0\\
  0&0&0&1\\
  1&1&0&0\\
  0&1&1&0\\
  0&0&1&1\\
  A&0&0&B
\end{array}\right]
\]
The $70=\binom84$ degree 4 minors (in a lexicographic order) equal
$1$, $0$, $0$, $-1$, $B$, $0$, $-1$, $1$, $0$, $0$, $0$, $0$, $1$, $-B$, $B$, $-1,
1$, $0$, $0$, $0$, $-1$, $B$, $-1$, $B$, $0$, $-1$, $1$, $0$, $1$, $0$, $0$, $-1$, $B,
-B$, $B$, $1$, $0$, $0$, $-A$, $0$, $1$, $-B$, $0$, $0$, $-A$, $1$, $-1$, $0$, $0,
-A$, $A$, $1$, $-B$, $B$, $-A$, $-1$, $0$, $A$, $0$, $A$, $0$, $-1$, $B$, $-A$, $A,
1$, $-A$, $A$, $-A$, $A-B$. Incidence table is the following
list of 24 quadruples of indices\\
{
$1235$, $1236$, $1245$, $1248$, $1256$,
$1257$, $1258$, $1347$, $1348$, $1356$,
$1378$, $1458$, $1468$, $1478$, $2346$,
$2347$, $2356$, $2367$, $2368$, $2458$,
$2467$, $3457$, $3467$, $3478$}.  \\
Although the (sorted) incidence table is independent of the coordinates in
$\PP^{3}$, but it depends on the order of eight planes. We ran through
all the permutations of planes and found that the minimal incidence
tables is \\
$1234$, $1235$, $1236$, $1237$, $1238$,
$1245$, $1267$, $1345$, $1367$, $1456$,
$1457$, $1458$, $1468$, $1568$, $2345$,
$2367$, $2467$, $2468$, $2478$, $2567$,
$2678$, $3468$, $4568$, $4678$.

The \emph{minimizing permutation} is given by the following product of
disjoint cycles $(126834)(57)$ of eight letters.

Observe, that we can try to revert the above considerations. For an
incidence table write down a matrix with generic
coefficients  and compute the appropriate minors
resulting in a system of equations (in 16 variables) describing all
octic arrangements with the given incidence matrix. More precisely we
get an arrangement with incidence table containing the table we
started with, using computer algebra system we check if (and which)
remaining minors belongs to the ideal generated by the assigned ones. 

The disadvantage
of this direct approach is that we  get a very complicated
system of equations that would be very difficult to handle. To simplify
the computations we can identify a rigid subarrangement of five or six
planes. Many of the considered
arrangements contain a rigid subarrangement of six faces of a
cube (which reduces the number of
parameters to 8 and the degrees of the equations to 2). In remaining
cases we identify a generic subarrangement of five planes (12
parameters, degree of equations at most 3).

\section{Special elements}

In every non--trivial one parameter family of Calabi--Yau threefolds
there are special elements which are not smooth. In the case of a
double octic Calabi--Yau threefolds they corresponds to arrangements
with altered types (or numbers) of singularities. Equivalently they
are the arrangements with bigger incidence table, consequently they
are described by vanishing of all the non--zero minors of the
coefficients matrix. 

We can encounter two different situation
\begin{itemize}
\item special arrangement does not satisfy the definition of an octic
  arrangement (i.e. it does not obey the restrictions we put on
  the intersections --- two planes equals, four planes intersect in a
  line or six planes intersect), then the special element doe-snot
  admit a Calabi--Yau resolution.
\item special arrangement is an octic arrangement of different type
  (non--equivalent). The blow--ups we perform to resolve a generic
  element of the family are not enough to resolve them, special
  element of the family is again singular but this time it admits a
  Calabi--Yau resolution (Calabi--Yau variety). 
\end{itemize}

In \cite{CvS} a conifold expansion was used to compute the
Picard--Fuchs operator for families in which four planes in general
position degenerate to four intersecting planes without any further
degenerations -- geometrically it corresponds to  a shrinking
tetrahedron. This kind of degenerations occur when the incidence table
increases by a one quadruple. There can be more that one shrinking
tetrahedron, in this case there are more then one new incidence, but
any two of them have at most two common incidences.
   
\section{Automorphisms, elliptic fibrations}

Arithmetic properties of Calabi--Yau threefolds (f.i. in the context of
Modularity  Conjecture -- see  \cite{Meyer} for details) depend on
special geometric properties, we shall study  Kummer surface fibrations 
and actions of finite groups.

If the eight planes split into two
quadruples of intersecting planes (two opposite fourfold points),
then each quadruple produces an elliptic surface, the
double octic Calabi--Yau has a double cover by an fiber product of
rational elliptic surfaces. 
The double octic is a Kummer surface fibration for the
two elliptic fibrations. 
Consider an arrangement $D=P_{1}\cup\dots\cup P_{8}$ of eight planes given by equations
\[
P_{i}=\{f_{i}=0\}, i=1,\dots,8
\]
and assume that among the eight planes there are two disjoint quadruples
intersecting in a point each. After renumbering the equations and changing coordinates
we can assume that the planes $P_{1},\dots,P_{4}$ intersect at the
point $A=(0,1,0,0)$
whereas $P_{5},\dots,P_{8}$ intersect at $B=(1,0,0,0)$. 
Equivalently, equations $f_{1},\dots,f_{4}$ depend only on $x,z,t$, 
$f_{5},\dots,f_{8}$ depend only on $y,z,t$.
Equations $f_{1},\dots,f_{4}$ (resp. $f_{5},\dots,f_{8}$) define four
lines in the projective plane $\PP^{2}_{x,z,t}$
(resp. $\PP^{2}_{y,z,t}$), that we can identify with projection from
$A$ and $B$ respectively.

Let $S$ and $S'$ be the double coverings of $\PP^2$ branched along the
corresponding sums of four lines. Then in appropriate affine
coordinates (f.i., $t=1$) they can be written as follows:
\begin{align*}
S & = \{(x,z,u)\in \mathbb C^{3}:u^{2}=f_{1}(x,z,1)\cdot\ldots\cdot f_{4}(x,z,1)\}\\
S' & = \{(y,z,v)\in \mathbb C^{3}:v^{2}=f_{5}(y,z,1)\cdot\ldots\cdot f_{8}(y,z,1)\}
\end{align*}

This exhibits (birationally) both surfaces as elliptic fibrations. Moreover,
the map
\[
((x,y,u),(y,t,v)) \mapsto (x,y,t,uv)
\]
is a rational, generically $2:1$ map from their fiber product to
the double covering of $\PP^3$ branched along the octic surface $D$.

Singular fibers of the elliptic surfaces $S$ ($S'$) correspond to the
projection of the six lines $l_{ij}=P_{i}\cap P_{j}$ with $1\le i<j\le
4$ ($m_{ij}=P_{i+4}\cap P_{j+4}$) from the line $AB$, we call the
lines $l_{ij}$ and $l_{jk}$ (resp. $m_{ij}$ and $m_{jk}$)
\emph{conjugate} if $\{i,j,k,l\}=\{1,2,3,4\}$. 
The type of a singular fiber is determined by the number of lines that project to the same
point
\begin{itemize}
\item [$I_{2}$] in the case of one line,
\item [$I_{4}$] -- two lines,  
\item [$I_{o}^{*}$] -- three lines,  
\item [$I_{2}^{*}$] -- four lines.  
\end{itemize}
Consequently, the fibers of type $I_{0}^{*}$ and $I_{2}^{*}$ can be
recognized from the incidence table. A fiber of the type $I_{0}$ can
occur in two different situation, three of the planes
$P_{1},\dots,P_{4}$ intersect along a (triple) line (so the three
lines actually coincide), one of the planes $P_{1},\dots,P_{4}$ pass
through the (fivefold) point $B$. A fiber of the type $I_{2}^{*}$
corresponds to three planes (out of $P_{1},\dots,P_{4}$) intersecting
along a line, and one of them passing through the point $B$. 
Consequently we can get surfaces with the following sequences of
singular fibers
\[\begin{array}{rlrlrl}
S_{1}:&(I_{0}^{*},I_{0}^{*}),
&S_{2}:&(I_{4},I_{4},I_{2},I_{2}), 
&S_{3}:&(I_{0}^{*},I_{2},I_{2},I_{2}),\\
S_{4}:&(I_{2}^{*},I_{2},I_{2}),\quad 
&S_{5}:&(I_{4},I_{2},I_{2},I_{2},I_{2}),\quad 
&S_{6}:&(I_{2},I_{2},I_{2},I_{2},I_{2},I_{2}).\quad 
\end{array}
\]

In table~\ref{tab:es} we give examples of explicit equations for
the branch locus of corresponding double quartic elliptic fibrations
and types and coordinates of the singular fibers.
\begin{table}
\[
\def\arraystretch{1.5}
\def\arraycolsep{2mm}
\begin{array}[t]{|c|cccccc|c|}
 \hline
\def\multicolumn#1#2{}
 \quad S_{1} \quad&I_{0}^{*}&I_{0}^{*}&&&&&{xz(x+t)(x+2t)}\\
 \cline{2-7}
 &\infty&0&&&&&\\
\hline
 S_{2}&I_{4}&I_{4}&I_{2}&I_{2}&&&{x(x+t)(x+z)(x+z+t)}\\
 \cline{2-7}
 &\infty&0&1&-1&&&\\
\hline
 S_{3}&I_{0}^{*}&I_{2}&I_{2}&I_{2}&&&{x(x+t)(x+\lambda t)(x+z)}\\
 \cline{2-7}
 &\infty&0&1&\lambda&&&{t(x+\lambda z)(x+z)(x+\lambda t)}\\
\hline
 S_{4}&I_{2}^{*}&I_{2}&I_{2}&&&&{xt(x+z)(x+t)}\\
 \cline{2-7}
 &\infty&0&1&&&&\\
\hline
S_{5}&I_{4}&I_{2}&I_{2}&I_{2}&I_{2}&&{x(x+t)(x+z-\lambda
  t)(x+z)}\\
 \cline{2-7}
 &\infty&0&1&\lambda&\lambda+1&&\\
\hline
S_{6}&I_{2}&I_{2}&I_{2}&I_{2}&I_{2}&I_{2}&{x(x+t)(x+z)(x+\frac{1}{\mu-\lambda}
  (z-\lambda t))}\rule[-3mm]{0pt}{8mm}\\
 \cline{2-7}
 &\infty&0&1&\lambda&\mu&\frac{\lambda}{\lambda-\mu+1}&
  x(x+t)(\lambda x+z)((\mu-1)x+z-t)\rule[-3mm]{0pt}{8mm}\\
\hline
\end{array}
\]
\caption{Elliptic fibrations}
\label{tab:es}
\end{table}
Fibration $S_{1}$ is special as it has only two singular fibers. For
the remaining cases we specialize three of singular fibers to
$\infty$, $0$ and $1$. Surfaces $S_{2}$ and $S_{4}$ are uniquely fixed by the position of
singular fibers. Singular fibers of $S_{2}$  form a harmonic
quadrilateral, there is an isogeny of this 
elliptic surface 
that exchanges $I_{2}$ fibers  with $I_{4}$ fibers (yielding some correspondences
between certain double octics).
For any fixed position of singular fibers there exists two non--equivalent
fibrations of type $S_{3}$ corresponding to the two described above
geometric realization. The following transformation
\[(u,x,z,t)\mapsto (\lambda ztu,\lambda zt,xz,xt)\]
gives a birational transformation of the two elliptic surface given in
the table, and consequently birational transformations of some double octics coming
from non-equivalent octic arrangements.\label{sec:birat}
A surface of type $S_{5}$ has one $I_{4}$ fiber and two pairs of
conjugate $I_{2}$ singularities, for any symmetric position of
singular fibers there exists a unique surface of type $S_{5}$. 

Surface of type $S_{6}$ has three pairs of conjugate $I_{2}$
singularities symmetric with respect to some involutive
automorphism of the projective line . For any symmetric position of
six singular fibers there 
are two surfaces of type $S_{6}$, geometrically we can say that they
are keeping track of the order of the conjugate pairs, we can change
order of any two of them. If we specialize three pairwise
non--conjugate singular fibers at $\infty$, $0$ and $1$ we have two
options: corresponding three lines $l_{ij}$ are -- up to permutation
-- $l_{12}, l_{13}, l_{23}$ or $l_{12},l_{1,3},l_{1,4}$ leading to the
two equations in the table. These two surfaces are birational and the
birational map can be given by 
\[(u,x,z,t)\longmapsto\left(u(t-z)(\lambda
  t-z)\sqrt{\tfrac{\lambda-\mu}{\lambda(\mu-1)}},
  (t-z)x,(\lambda x-x+\lambda t-z)z,(\lambda x-x+\lambda t-z)t\right)
\]

Fiber product of rational elliptic surfaces (hence also the Kummer
fibration) is determined by the matching of singular fibers, singular
fibers given by the lines $l_{ij}$ and $m_{kl}$ are mapped to the same
point in $\PP^{1}$ iff the lines $l_{ij}$ and $m_{kl}$ intersects
i.e. the corresponding four planes intersect in a point (the
appropriate quadruple belongs to the incidence table).

For every considered octic arrangement we determine permutations in
$S_{8}$ which leave the incidence table invariant, using MAGMA
(and Gap's SmallGroups Library) we classify the group of symmetries
and small system of generators. Using this data one can
determine which symmetries lift to projective transformation of
$\PP^{3}$ and the double octic.

\section{Algorithm}

Any incidence table determines two additional lists: the list of
triplets of planes intersecting along a (triple) line and the list of
quintuples of planes intersecting at a (fivefold) point. A triplet
belongs to the first list iff all quadruples containing it belongs to
the incidence table, similarly a quintuple belongs to the second list
if all quadruples contained in it belongs to the incidence table. 

If two quadruples belonging to the incidence table have three common
interface then their intersection belongs to the list of triplets or
their sum belongs to the list of quintuples (or both). Finally neither
two triplets can have two common entries nor two quintuples -- four
common entries. Finally the lists of triplets and quintuples have at
most 4 elements each (an arrangement has at most four triple lines and
at most four fivefold points).  

Our strategy is to operate on a list which entries consist of three
list: list of triplets, quadruples and quintuples.
We start with the  list with one entry: the list of quadruples contains one
element $1234$ (the smallest possible), the lists of triplets
and quintuples are empty. Then we repeat the following steps on each
entry of the list

\begin{itemize}
\item introduce new element in the list by adding a single quadruple
  (that is not yet on a list of quadruples),
\item if two quadruples in the list have three common entries we
  replace corresponding entry with two new ones by adding the
  intersection to the triplets list or the sum to the quintuples list
\item it two triplets have two or quintuples -- four common entries,
  we remove the corresponding element from the list
\item for each triplets we add all quadruples containing it and for
  each quintuple -- all quadruples contained in it
\item we apply all permutations of eight digits to the triplets,
  quadruples and quintuples list and replace the entry with minimal
  list of quadruples.
\end{itemize}

We repeated these steps until the list stabilized. Then we applied the
inverse procedure to the entries which were obtained after first seven
steps, which is a necessary condition for $h^{1,2}\le1$. For each such
entry we determined a quintuples of planes in general position in the
arrangement (i.e. a quintuple of eight digits which does not contain a
quadruple from the incidence table). Then we assumed that these five
planes have equations $x,y,z,t,x+y+z+t$, whereas the remaining three
planes have equations $A_{i}x+B_{i}y+C_{i}x+D_{i}t$ and computed the
ideal generated by the minors of entries of the coefficients matrix
corresponding to the entries in the incidence table. Associated primes
of this ideal that do not impose any extra elements in the incidence
table give us the requested octic arrangements.

Finally we computed the minimal incidence tables for the examples from
\cite{Meyer} and recognize them in the produced list.

\section{Results}

In this section we collect the results of Magma computations, for the
sake of completeness we also include the data of the eleven rigid
arrangements listed in \cite{Meyer} and three more that we shall
denote as A, B and C. Observe that these three examples are not
defined over $\QQ$ so they cannot be found in \cite{Meyer}, however
they appeared in \cite{CvS, CvS2}. Our computations produced only one new
example which we denote by D, a family defined over $\QQ[\sqrt{-3}]$,
again this example could not appear in the list of Meyer. Finally,
there is only one rigid double octic Calabi--Yau threefold in positive
characteristic, that cannot be lifted to characteristic zero, it is
the non--liftable example in characteristic 3 in \cite{CvS3}.
\begin{prop}
  The lists of  double octic Calabi--Yau threefolds with $h^{1,2}\le1$
  in \cite{Meyer} defined over $\QQ$ is  complete, there exist four
  examples that cannot be realized over $\QQ$, they have the following
  Hodge numbers:
  \begin{eqnarray*}
    &&A: h^{1,1}=46, h^{1,2}=0\\
    &&B: h^{1,1}=38, h^{1,2}=0\\
    &&C: h^{1,1}=38, h^{1,2}=0\\
    &&D: h^{1,1}=36, h^{1,2}=1
  \end{eqnarray*}
\end{prop}

Most of the data are self explaining, we do not include the
(original) incidence tables as they can be easily computed.
We also do not include the full symmetry groups, instead we give an
isomorphism type and a small set of generators. In fact we get 16
isomorphism types, most of them can be represented as direct products
of cyclic, dihedral and symmetric groups, three biggest and most
complicated we denoted by $g_{32,43}$, $g_{64,138}$ and $g_{192,955}$
(notation follows the SmallGroups library from GAP). 
These groups can be described as:

\begin{itemize}\leftskip=6mm
\item[\rule{3mm}{0mm}$g_{32,43}$:\ ] is the holomorph of the cyclic group
  $C_{8}$, i.e. $\operatorname{Aut}(C_{8})\rtimes C_{8}$;
\item[\rule{1.5mm}{0mm}$g_{64,138}$:\ ] is the unitriangular matrix group
 $UT(4,2)$  of degree four over the field of two elements. 
\item[$g_{192,955}$:\ ] is a semidirect product $C_{2}^{\oplus4}\rtimes
  D_{6}$ and is isomorphic to the automorphism group
  $\operatorname{Aut}(C_{2}\times Q_{8})$
\end{itemize}

Some arrangements have several pairs of opposite fourfold points and
so also corresponding fiber products of rational elliptic surfaces,
we list only one example for each isomorphisms class. 
In case of surfaces of types $S_{5}$ and $S_{6}$ we marked pairs of conjugate
singular fibers. 

\vspace*{0mm plus 6mm}

\subsection{Rigid arrangements}

\def\wl#1{\hbox to 10mm{\hfill$\ifcase#1\or I_{2}\or I_{4}\or I_{0}^{*}\or I_{2}^{*}\fi$\hfill}}

\long\def\AAA#1

#2

#3

#4

#5

#6

#7

#8

#9

{\vspace{2mm plus 2mm}{\small
\noindent
\textbf{Arr. No. #1}:
\(\displaystyle #2\)

\nopagebreak
\noindent\textbf{Minimal incidences: }
#3

\nopagebreak
\noindent
\textbf{Minimizing permutation: }$#4$

\nopagebreak
\ifthenelse{\equal{#5}{1,1}}\relax
{\noindent {\textbf{Symmetries:}} $#5,\quad \langle 
#6
\rangle$}

\vbox{\nopagebreak
{\noindent {\textbf{Singular points:}\par} 
#9
}\\[-2mm]}

\nopagebreak
\ifx\\#8\\\relax\else{\noindent{\textbf{Elliptic fibrations:}}
#8} 
\fi\par\noindent
}
\vspace*{-1mm}
}

\AAA 1

xyzt \left( x+y \right)  \left( y+z \right)  \left( z+t \right) 
 \left( x+t \right) 

$1234$, $1235$, $1236$, $1237$, $1238$,
$1245$, $1267$, $1345$, $1367$, $1456$,
$1457$, $1458$, $1468$, $1568$, $2345$,
$2367$, $2467$, $2468$, $2478$, $2567$,
$2678$, $3468$, $3578$, $4568$, $4678$

(126834)(57)

D_4

{(1432)(5876),(12)(34)(68)}

{\hbox{1256 -- 3478,}\;\;\hbox{1258 -- 3467,}\;\;\hbox{1458 -- 2367,}\;\;\hbox{1478 -- 2356,}\;\;}

\qquad\(\displaystyle
\begin{array}[t]{ccc}
  \infty&0&1\\\hline\wl4 & \wl1 &\wl1 \\ \wl1 &\wl4&\wl1
\end{array}
\)\\

$p^0_4:$  
5678

$p^1_4:$  
1275, 1468, 2386, 3457

$p^2_5:$  
12356, 12458, 13478, 23467

$\ l_3:$
125, 148, 236, 347

\AAA 3

xyzt \left( x+y \right)  \left( y+z \right)  \left( y-t \right) 
 \left( x-y-z+t \right) 

$1234$, $1235$, $1236$, $1237$, $1238$, $1245$, $1267$, $1345$,
$1367$, $1456$, $1457$, $1458$, $1467$, $1567$, $1678$, $2345$,
$2367$, $2468$, $2578$, $3478$, $4567$ 

(172)(3564)

S_3

{(176)(354),(17)(45)}

{\hbox{1258 -- 3467,}\;\;\hbox{1356 -- 2478,}\;\;\hbox{1457 -- 2368,}\;\;}

\qquad\(\displaystyle
\begin{array}[t]{cccc}
  \infty&0&1&2\\\hline\wl4 &-& \wl1 &\wl1 \\ \wl2 &\wl1&\wl2&\wl1
\end{array}
\)\\

$p^0_4:$  
1378, 1468, 5678

$p^1_4:$  
1285, 2386, 2487

$p^2_5:$  
12356, 12457, 23467

$\ l_3:$
125, 236, 247

\AAA 19

xyzt \left( x+y \right)  \left( y+z \right)  \left( x-z-t \right) 
 \left( x+y+z-t \right) 

$1234$, $1235$, $1236$, $1237$, $1238$,
$1245$, $1267$, $1345$, $1367$, $1456$,
$1457$, $1458$, $2345$, $2367$, $2468$,
$2578$, $3478$, $3568$,  

(1548632)

C_2\oplus C_2

{(15)(78),(15)(36)}

{\hbox{1245 -- 3678,}\;\;}

\qquad\(\displaystyle
\begin{array}[t]{cccc}
  \infty&0&1&1/2\\\hline\wl4 &\wl1& \wl1 &- \\ \wl1 &\wl2&\wl2&\wl1
\end{array}
\)\\

$p^0_4:$  
1347, 1468, 3458, 4567

$p^1_4:$  
1245, 1275, 1285, 2346

$p^1_5:$  
23768

$p^2_5:$  
12356

$\ l_3:$  
125, 236

\AAA 32

xyzt \left( x+y \right)  \left( y+z \right)  \left( x-y-z-t \right) 
 \left( x+y-z+t \right) 

$1234$, $1235$, $1236$, $1237$, $1238$,
$1245$, $1345$, $1456$, $1457$, $1458$,
$1678$, $2345$, $2467$, $2568$, $3468$,
$3578$,  

(15478632)

C_2\oplus C_2

{(13)(56)(78),(16)(35)}

{\hbox{1356 -- 2478,}\;\;}

\qquad\(\displaystyle
\begin{array}[t]{cccc}
  \infty&0&1&-1\\\hline\wl2 &\wl2& \wl1 &\wl1 \\ \wl3 &\wl1&\wl1&\wl1
\end{array}
\)\\

$p^0_4:$  
1378, 1467, 2478, 3458, 5678

$p^1_4:$  
1245, 1275, 1285, 2346, 2376, 2386

$p^2_5:$  
12356

$\ l_3:$  
125, 236

\AAA 69

xyzt \left( x+y \right)  \left( x-y+z \right)  \left( x-y-t \right) 
 \left( x+y-z-t \right) 

$1234$, $1235$, $1236$, $1237$, $1238$, $1245$, $1267$, $1345$, $1367$, $1468$, $1578$, $2345$, $2367$, $2478$, $3568$, $4567$

(13745)

C_2\oplus C_2

{(12)(34)(67),(12)(36)(47)}

{\hbox{1258 -- 3467,}\;\;}

\qquad\(\displaystyle
\begin{array}[t]{cccc}
  \infty&0&1&-1\\\hline\wl2 &\wl2& \wl1 &\wl1 \\ \wl3 &\wl1&\wl1&\wl1
\end{array}
\)\\

$p^0_4:$  
1468, 2378, 3458, 3467, 5678

$p^1_4:$  
1285

$p^1_5:$  
12356, 12457

$\ l_3:$  
125

\AAA 93

xyzt \left( x+y \right)  \left( x-y+z \right)  \left( y-z-t \right) 
 \left( x+z-t \right) 

$1234$, $1235$, $1236$, $1237$, $1238$, $1245$, $1345$, $1467$, $1568$, $2345$, $2468$, $2567$, $3678$, $4578$ 

(12)(35)(46)(78)

C_2\oplus C_2

{(12)(36),(12)(78)}

{\hbox{1236 -- 4578,}\;\;\hbox{1245 -- 3678,}\;\;}

\qquad\(\displaystyle
\begin{array}[t]{cccc}
  \infty&0&1&1/2\\\hline\wl1 &\wl2& \wl2 &\wl1 \\ \wl3 &\wl1&\wl1&\wl1
\end{array}
\)\\

$p^0_4:$  
1348, 1467, 2347, 2468, 3678, 4578

$p^1_4:$  
1245, 1275, 1285

$p^1_5:$  
12356

$\ l_3:$  
125

\AAA 238

xyzt \left( x+y+z-t \right)  \left( x+y-z+t \right)  \left( x-y+z+t
 \right)  \left( -x+y+z+t \right) 

$1234$, $1256$, $1278$, $1357$, $1368$, $1458$, $2367$, $2457$, $2468$, $3456$, $3478$, $5678$ 

(172)(346)

G_{192,955}

{(23)(5687),(1835)(2746)}

{\hbox{1256 -- 3478,}\;\; \hbox{1278 -- 3456,}\;\; \hbox{1357 --
    2468,}\;\; \hbox{1368 -- 2457,}\;\; \hbox{1458 -- 2367,}\;\; \hbox{1467 -- 2358,}\;\;}

\qquad\(\displaystyle
\begin{array}[t]{cccc}
  \infty&0&1&-1\\\hline\wl1 &\wl1& \wl2 &\wl2 \\ \wl1 &\wl1&\wl2&\wl2
\end{array}
\)\\

$p^0_4:$  
1256, 1278, 1357, 1368, 1458, 1467, 2358, 2367, 2457, 2468, 3456, 3478

\AAA 239

xyzt \left( x+y+z \right)  \left( x+y+t \right)  \left( x+z+t \right) 
 \left( y+z+t \right) 

$1234, 1256, 1278, 1357, 1368, 1458, 2358, 2367, 2457, 3456$ 

(13)(45)(68)

S_4

{(23)(67),(1234)(5876)}

{\hbox{1278 -- 3456,}\;\;\hbox{1368 -- 2457,}\;\;\hbox{1458 -- 2367,}\;\;}

\qquad\(\displaystyle
\begin{array}[t]{ccccc}
  \infty&0&-1/2&-1&-2\\\hline\wl1 &\wl2& \wl1 &\wl2&- \\ \wl2 &\wl1&-&\wl2&\wl1
\end{array}
\)\\

$p^0_4:$  
1235, 1246, 1278, 1347, 1368, 1458, 2348, 2367, 2457, 3456

\AAA 240

xyzt \left( x+y+z \right)  \left( x+y-z+t \right)  \left( x-y+z+t
 \right)  \left( x-y-z-t \right) 

$1234$, $1256$, $1278$, $1357$, $1368$, $1458$, $2367$, $2457$, $3456$, $4678$ 

(14)(27)(36)(58)

S_3

{(12)(78),(123)(687)}

{\hbox{1278 -- 3456,}\;\;\hbox{1368 -- 2457,}\;\;\hbox{1458 -- 2367,}\;\;}

\qquad\(\displaystyle
\begin{array}[t]{ccccc}
  \infty&0&1&-1&-3\\\hline\wl1 &\wl1& \wl2 &\wl2&- \\ \wl1
        &\wl1&\wl1&\wl2&\wl1\\[-2mm]
&\multicolumn{4}{c}{\rule{.5pt}{2mm}\rule{4cm}{.5pt}\rule{.5pt}{2mm}}\\[-3mm]
\multicolumn{3}{c}{\rule{.5pt}{3mm}\rule{2.8cm}{.5pt}\rule{.5pt}{3mm}}
\end{array}
\)\\

$p^0_4:$  
1235, 1278, 1368, 1458, 1467, 2367, 2457, 2468, 3456, 3478

\AAA 241

xyzt \left( x+y+z+t \right)  \left( x+y-z-t \right)  \left( y-z+t
 \right)  \left( x+z-t \right) 

$1234, 1256, 1278, 1357, 1468, 2358, 2467, 3456, 3678, 4578$

(13267584)

G_{32,43}

{(34)(56),(17)(36)(45),(14857326),(28)(36)(45)}

{\hbox{1278 -- 3456,}\;\;}

\qquad\(\displaystyle
\begin{array}[t]{cccc}
  \infty&0&1&-1\\\hline\wl2 &\wl2& \wl1 &\wl1 \\ \wl1 &\wl1&\wl2&\wl2
\end{array}
\)\\

$p^0_4:$  
1256, 1278, 1348, 1357, 1467, 2347, 2368, 2458, 3456, 5678

\AAA 245

xyzt \left( x+y+z \right)  \left( y+z+t \right)  \left( x-y-t \right) 
 \left( x-y+z+t \right) 

$1234, 1256, 1278, 1357, 1368, 1458, 2358, 2367, 4567$

(1362)(458)

C_2

{(16)(37)}

{\hbox{1367 -- 2458,}\;\;}

\qquad\(\displaystyle
\begin{array}[t]{ccccc}
  \infty&0&1&-1&-3\\\hline\wl2 &\wl2& \wl1 &\wl1&- \\ \wl1
        &\wl1&\wl1&\wl2&\wl1\\[-2mm]
&\multicolumn{4}{c}{\rule{.5pt}{2mm}\rule{4cm}{.5pt}\rule{.5pt}{2mm}}\\[-3mm]
\multicolumn{3}{c}{\rule{.5pt}{3mm}\rule{2.8cm}{.5pt}\rule{.5pt}{3mm}}
\end{array}
\)\\

$p^0_4:$  
1235, 1247, 1268, 1367, 1456, 2346, 2458, 2567, 3478

\AAA A

xyzt \left( x+y \right)  \left( x+y+z-t \right) \times\\\rule{2cm}{0cm} \times \left(  \left( \sqrt{-3}-1
 \right) x-2\,y+ \left( \sqrt{-3}-1 \right) z \right)\left( 2\,y+ \left( -\sqrt{-3}+
1 \right) z-2\,t \right) 

$1234$, $1235$, $1236$, $1237$, $1238$, $1245$, $1345$, $1467$, $1568$, $2345$, $2468$, $2578$, $3478$, $3567$ 

(13468752)

S_3

{(12)(37)(48),(25)(37)(46)}

{}

{}

$p^0_4:$  
1368, 1478, 2348, 2467, 3456, 5678

$p^1_4:$  
1245, 1265, 1285

$p^1_5:$  
12357

$\ l_3:$  
125

\AAA B

xyzt \left( x+y+z \right)  \left( x+z-t \right)  \left(  \left( \sqrt{-3}-1
 \right) x+ \left( \sqrt{-3}+1 \right) y-2\,z+2\,t \right) \\\rule{5cm}{0mm}\times \left(  \left( \sqrt{-3}-1
 \right) x+ \left( \sqrt{-3}-1 \right) y-2\,z+ \left( \sqrt{-3}+1 \right) t \right) 

$1234, 1256, 1278, 1357, 1368, 2358, 2457, 3456, 4678$ 

(13528476)

D_4

{(28)(36),(16)(27)(35)(48)}

{\hbox{1267 -- 3458,}\;\;\hbox{1378 -- 2456,}\;\;}

\[\def\arraystretch{1.2}
\begin {array}{ccccc} 0&1&\infty& \tfrac12-\tfrac{\sqrt{-3}}2&
-\tfrac12-\tfrac{\sqrt{-3}}2\\[2mm] 
\hline  \wl1&\wl1&\wl2&\wl1&\wl1\\ [-2mm]
\multicolumn{4}{c}{\rule{.5pt}{2mm}\rule{4.2cm}{.5pt}\rule{.5pt}{2mm}\rule{2mm}{0mm}}\\[-4mm]
&\multicolumn{4}{c}{\rule{.5pt}{3mm}\rule{4.8cm}{.5pt}\rule{.5pt}{3mm}\rule{4mm}{0mm}}\\[-1mm]
\wl2&\wl1&\wl1&\wl1&\wl1\\ [-2mm]
&&\multicolumn{2}{c}{\rule{.5pt}{2mm}\rule{1.6cm}{.5pt}\rule{.5pt}{2mm}\rule{2mm}{0mm}}\\[-4mm]
&\multicolumn{4}{c}{\rule{.5pt}{3mm}\rule{4.8cm}{.5pt}\rule{.5pt}{3mm}\rule{4mm}{0mm}}\\
\end {array} 
\]

$p^0_4:$  
1235, 1267, 1346, 1378, 1568, 2456, 2478, 3458, 3567

\AAA C

xyzt \left( x+y+z \right)  \left(  \left( \sqrt5-1 \right) y-2\,z+2\,t
 \right)  \left( 2\,x+2\,y+ \left( \sqrt5-1 \right) t \right)\\\rule{5cm}{0mm}\times  \left( 
 \left( -\sqrt5+3 \right) x+2\,y+ \left( -\sqrt5+1 \right) z+ \left( \sqrt5-1
 \right) t \right) 

$1234, 1256, 1278, 1357, 1368, 2358, 2457, 3467, 4568$ 

(125387)(46)

C_2

{(17)(25)(38)(46)}

{\hbox{1268 -- 3457,}\;\;}

\[\def\arraystretch{1.2}
\begin {array}{ccccc} 0&1&\infty& -\tfrac12-\tfrac{\sqrt{5}}2&
                                                              \tfrac12-\tfrac{\sqrt{5}}2\\[2mm] 
\hline  \wl1&\wl1&\wl1&\wl1&\wl2\\ [-2mm]
&\multicolumn{2}{c}{\rule{.5pt}{2mm}\rule{1.4cm}{.5pt}\rule{.5pt}{2mm}\rule{2mm}{0mm}}\\[-4mm]
\multicolumn{4}{c}{\rule{.5pt}{3mm}\rule{4.cm}{.5pt}\rule{.5pt}{3mm}}\\[-1mm]
\wl1&\wl1&\wl2&\wl1&\wl1\\ [-2mm]
&\multicolumn{3}{c}{\rule{.5pt}{2mm}\rule{3cm}{.5pt}\rule{.5pt}{2mm}\rule{2mm}{0mm}}\\[-4mm]
\multicolumn{5}{c}{\rule{.5pt}{3mm}\rule{5.8cm}{.5pt}\rule{.5pt}{3mm}}\\
\end {array} 
\]

$p^0_4:$  
1235, 1247, 1268, 1378, 1567, 2346, 2578, 3457, 4568

\subsection{Non--rigid arrangements}
\rule{0mm}{0mm}

\def\nc{\  non--CY}
\def\arr{~\;~Arr.\,}
\def\inf{$\infty$}
\def\qdd{\rule{5mm}{0mm}}

\long\def\AAA#1

#2

#3

#4

#5

#6

#7

#8

#9

{\vspace{2mm plus 2mm minus 2mm}\vbox{{\small
\noindent
\textbf{Arr. No. #1}:
\(\displaystyle #2\)

\nopagebreak
\noindent\textbf{Minimal incidences: }
#3

\nopagebreak
\noindent
\textbf{Minimilizing permutation: }$#4$

\nopagebreak
\ifthenelse{\equal{#5}{1,1}}\relax
{\noindent \textbf{Symmetries: } $#5,\quad \langle 
#6
\rangle$}

\nopagebreak
\noindent\textbf{Special values: }#8
\par\noindent
\vspace*{-5mm}
}}

\nopagebreak
{\noindent {\textbf{Singular points:}\par} 
#9
}\par\noindent
}

\long\def\ELLF#1

{\nopagebreak
\ifx\\#1\\\relax\else{\vspace*{4mm}\noindent\textbf{Elliptic fibrations:}
#1} 
\fi\par\noindent
}

\AAA 2

xyzt \left( x+y \right)  \left( y+z \right)  \left( z+t \right) 
 \left( Ax+Bt \right) 

$1234$, $1235$, $1236$, $1237$, $1238$, $1245$, $1267$, $1345$, $1367$, $1456$, $1457$, $1458$, $1468$, $1568$, $2345$, $2367$, $2467$, $2468$, $2478$, $2567$, $2678$, $3468$, $4568$, $4678$

(126834)(57)

D_4

{(1432)(5876),(14)(23)(57)}

{\hbox{1256 -- 3478,}\;\; \hbox{1258 -- 3467,}\;\; \hbox{1458 --
    2367,}\;\; \hbox{1478 -- 2356,}\;\;}

{\inf:\nc,\qdd 0:\nc,\qdd 1:\arr1,\qdd }

$p^1_4:$  
1275, 1468, 2386, 3457

$p^2_5:$  
12356, 12458, 13478, 23467

$\ l_3:$  
125, 148, 236, 347

\ELLF\qquad\(\displaystyle\def\arraystretch{1.2}
\begin{array}[t]{cccc}
  \infty&0&1&\frac AB\\[1mm]\hline\wl4 & \wl1 &\wl1&- \\ \wl1 &\wl4&-&\wl1
\end{array}
\)\\

\AAA 4

xyzt \left( x+y \right)  \left( y+z \right)  \left( Ax+By+Bz-At
 \right)  \left( Ax+Ay+Bz-At \right) 

$1234$, $1235$, $1236$, $1237$, $1238$, $1245$, $1267$, $1345$, $1367$, $1456$, $1457$, $1458$, $1467$, $1567$, $1678$, $2345$, $2367$, $2468$, $3578$, $4567$

(136752)(48)

D_6

{(137568),(16)(35)}

{\hbox{1245 -- 3678,}\;\;\hbox{1356 -- 2478,}\;\;\hbox{1578 -- 2346,}\;\;}

{\inf:\nc,\qdd 0:\nc,\qdd 1:\nc,\qdd }

$p^0_4:$  
1467, 3458

$p^1_4:$  
1245, 2346, 2478

$p^2_5:$  
12356, 12758, 23768

$\ l_3:$  
125, 236, 278

\ELLF\qquad\(\displaystyle\def\arraystretch{1.2}
\begin{array}[t]{ccccc}
  \infty&0&1&\frac BA&\frac{A-B}{A}\\[1mm]
\hline\wl4 &\wl1& \wl1 &-&- \\ \wl2 &\wl1&\wl1&\wl1&\wl1\\[-3mm]
&\multicolumn{2}{c}{\rule{.5pt}{2mm}\rule{1.3cm}{.5pt}\rule{.5pt}{2mm}}
&\multicolumn{2}{c}{\rule{.5pt}{2mm}\rule{1.3cm}{.5pt}\rule{.5pt}{2mm}}
\end{array}
\)\\

\AAA 5

xyzt \left( x+y \right)  \left( y+z \right)  \left( x+y+z-t \right) 
 \left( Ax+By+Az-At \right) 

$1234$, $1235$, $1236$, $1237$, $1238$, $1245$, $1267$, $1345$,
$1367$, $1456$, $1457$, $1458$, $1467$, $1567$, $1678$, $2345$,
$2367$, $2468$, $2578$, $4567$ 

(172)(348)(56)

C_2\oplus C_2

{(16)(35),(13)(56)}

{\hbox{1245 -- 3678,}\;\;\hbox{1356 -- 2478,}\;\;\hbox{1578 -- 2346,}\;\;}

{\inf:\arr3,\qdd 0:\nc,\qdd 1:\nc,\qdd 1/2:\arr3,\qdd }

$p^0_4:$  
1467, 3457

$p^1_4:$  
1245, 2346, 2478

$p^2_5:$  
12356, 12758, 23768

$\ l_3:$  
125, 236, 278

\ELLF\qquad\(\displaystyle
\begin{array}[t]{ccccc}
  \infty&0&1&\frac BA&-\frac{A-B}{A}\\[1mm]
\hline\wl2 &\wl1& \wl1 &\wl1&\wl1 \\[-2mm]
&\multicolumn{3}{c}{\rule{.5pt}{2mm}\rule{2.6cm}{.5pt}\rule{.5pt}{2mm}}\\[-4mm]
&&\multicolumn{3}{c}{\rule{.5pt}{3mm}\rule{2.8cm}{.5pt}\rule{.5pt}{3mm}}\\[-1mm] \wl4 &\wl1&\wl1&-&-
\end{array}
\)\\[3mm]
\rule{32mm}{0mm}\textbf{and}\qquad\(\displaystyle
\begin{array}[t]{ccccc}
  \infty&-2&-1&0&-\frac BA\\[1mm]
\hline\wl2 &\wl1& \wl2 &\wl1&- \\ \wl4 &-&\wl1&-&\wl1
\end{array}
\)\\

\AAA 8

xyzt \left( x+y \right)  \left( y+z \right)  \left( z-t \right) 
 \left( Ax-By-Bz+Bt \right) 

$1234$, $1235$, $1236$, $1237$, $1238$, $1245$, $1267$, $1345$,
$1367$, $1456$, $1457$, $1458$, $1468$, $1568$, $2345$, $2367$,
$2467$, $2567$, $2678$, $3478$, $4568$ 

(147632)

1,1

{}

{\hbox{1256 -- 3478,}\;\;\hbox{1258 -- 3467,}\;\;\hbox{1578 -- 2346,}\;\;}

{\inf:\nc,\qdd 0:\nc,\qdd -1:\arr1,\qdd }

$p^0_4:$  
1468

$p^1_4:$  
1245, 1347, 2386, 3457, 3487

$p^1_5:$  
12758

$p^2_5:$  
12356, 23467

$\ l_3:$  
125, 236, 347

\ELLF\qquad\(\displaystyle\def\arraystretch{1.2}
\begin{array}[t]{cccc}
  \infty&-1&0&-\frac B{A+B}\\[1mm]
\hline\wl3 &\wl1& \wl1 &\wl1\\ \wl1 &\wl1&\wl4&-
\end{array}
\)\\[3mm]
\rule{32mm}{0mm}\textbf{and}\qquad\(\displaystyle
\begin{array}[t]{cccc}
  \infty&0&1&\frac {A+B}B\\[1mm]
\hline\wl4 &-& \wl1 &\wl1\\ \wl1 &\wl4&\wl1&-
\end{array}
\)\\

\AAA 10

xyzt \left( x+y \right)  \left( y+z \right)  \left( z-t \right) 
 \left( Ax-By-Bz-At \right) 

$1234$, $1235$, $1236$, $1237$, $1238$, $1245$, $1267$, $1345$, $1367$, $1456$, $1457,
1458$, $2345$, $2367$, $2467$, $2567$, $2678$, $3468$, $3578$ 

(174563)

C_2\oplus C_2

{(15)(47),(17)(23)(45)}

{\hbox{1256 -- 3478,}\;\;\hbox{1258 -- 3467,}\;\;}

{\inf:\arr1,\qdd 0:\nc,\qdd -1:\arr1,\qdd }

$p^0_4:$  
1468, 5678

$p^1_4:$  
1245, 1275, 1285, 1347, 2386, 3457, 3487

$p^2_5:$  
12356, 23467

$\ l_3:$  
125, 236, 347

\ELLF\qquad\(\displaystyle\def\arraystretch{1.2}
\begin{array}[t]{cccc}
  \infty&0&-1&-\frac B{A+B}\\[1mm]
\hline\wl3 &\wl1& \wl1 &\wl1\\ - &\wl4&\wl1&\wl1
\end{array}
\)

\AAA 13

xyzt \left( x+y \right)  \left( y+z \right)  \left( x-z-t \right) 
 \left( Ax-Az+Bt \right) 

$1234$, $1235$, $1236$, $1237$, $1238$, $1245$, $1267$, $1345$, $1367$, $1468$, $1578,
2345$, $2367$, $2468$, $2578$, $3468$, $3578$, $4567$, $4568$, $4578$, $4678$, $5678$

(1728)(364)

S_3\oplus C_2\oplus C_2

{(13)(478)(56),(78),(15)(36),(478)}

{\hbox{1235 -- 4678,}\;\; \hbox{1236 -- 4578,}\;\; \hbox{1256 --
    3478,}\;\; \hbox{1356 -- 2478,}\;\; \hbox{1478 -- 2356,}\;\;}

{\inf:\nc,\qdd 0:\nc,\qdd -1:\nc,\qdd }

$p^1_4:$  
1245, 1275, 1285, 2346, 2376, 2386, 2478

$p^1_5:$  
13478, 45678

$p^2_5:$  
12356

$\ l_3:$  
125, 236, 478

\ELLF\qquad\(\displaystyle\def\arraystretch{1.2}
\begin{array}[t]{ccc}
  \infty&0&-1\\[1mm]
\hline\wl4 &\wl1& \wl1 \\ - &\wl3&\wl3
\end{array}
\)\\[3mm]
\rule{32mm}{0mm}\textbf{and}\qquad\(\displaystyle
\begin{array}[t]{cccc}
  \infty&-1&0&1\\[1mm]
\hline\wl2 &\wl1& \wl2 &\wl1\\ \wl3 &-&\wl3&-
\end{array}
\)\\

\AAA 16

xyzt \left( x+y \right)  \left( y+z \right)  \left( Ay-Bz-At \right) 
 \left( Bx-Ay+At \right) 

$1234$, $1235$, $1236$, $1237$, $1238$, $1245$, $1267$, $1345$, $1367$, $1456$, $1457,
1458$, $1468$, $1568$, $2345$, $2367$, $2478$, $3578$, $4568$

(1352)(46)(78)

C_2\oplus C_2

{(16)(35)(78),(15)(36)}

{\hbox{1258 -- 3467,}\;\;\hbox{1458 -- 2367,}\;\;}

{\inf:\nc,\qdd 0:\nc,\qdd -1:\arr1,\qdd }

$p^0_4:$  
1378, 5678

$p^1_4:$  
1275, 2386

$p^1_5:$  
12458, 23467

$p^2_5:$  
12356

$\ l_3:$  
125, 236

\ELLF\qquad\(\displaystyle\def\arraystretch{1.2}
\begin{array}[t]{cccc}
  \infty&0&1&\frac{A+B}A\\[1mm]
\hline\wl4 &-& \wl1 &\wl1\\ \wl1 &\wl3&\wl1&\wl1
\end{array}
\)

\AAA 20

xyzt \left( x+y \right)  \left( y+z \right)  \left( x-z+t \right) 
 \left( Ay-Bz-At \right) 

$1234, 1235, 1236, 1237, 1238, 1245, 1267, 1345, 1367, 1456, 1457,
1458, 2345, 2367, 2468, 2578, 3478$,  

(15478632)

1,1

{}

{\hbox{1257 -- 3468,}\;\;}

{\inf:\arr3,\qdd 0:\nc,\qdd -1:\arr1,\qdd -1/2:\arr19,\qdd }

$p^0_4:$  
1347, 3578, 4567

$p^1_4:$  
1245, 1275, 1285, 2376

$p^1_5:$  
23468

$p^2_5:$  
12356

$\ l_3:$  
125, 236

\ELLF\qquad\(\displaystyle
\begin{array}[t]{ccccc}
  \infty&0&1&\frac {2A+B}A&-\frac A{B}\\[1mm]
\hline\wl1 &\wl1& \wl2 &\wl1&\wl1 \\[-2mm]
&\multicolumn{3}{c}{\rule{.5pt}{2mm}\rule{2.6cm}{.5pt}\rule{.5pt}{2mm}}\\[-4mm]
\multicolumn{5}{c}{\rule{.5pt}{3mm}\rule{5.4cm}{.5pt}\rule{.5pt}{3mm}}\\[-1mm] \wl4 &\wl1&\wl1&-&-
\end{array}
\)\\

\AAA 21

xyzt \left( x+y \right)  \left( y+z \right)  \left( Ax-By+ \left( -A-B
 \right) t \right)  \left( Ax+Bz-At \right) 

$1234$, $1235$, $1236$, $1237$, $1238$, $1245$, $1267$, $1345$, $1367$, $1456$, $1457,
1458$, $2345$, $2367$, $2468$, $3478$, $5678$

(1347652)

C_2

{(15)(47)}

{\hbox{1235 -- 4678,}\;\;\hbox{1457 -- 2368,}\;\;}

{\inf:\nc,\qdd 0:\nc,\qdd -1:\nc,\qdd }

$p^0_4:$  
1348, 3578, 4678

$p^1_4:$  
1285, 2346, 2376, 2386

$p^1_5:$  
12457

$p^2_5:$  
12356

$\ l_3:$  
125, 236

\ELLF\qquad\(\displaystyle
\begin{array}[t]{ccccc}
  \infty&0&1&-\frac BA&-\frac{B}{A+B}\\[1mm]
\hline\wl1 &\wl2& \wl1 &\wl1&\wl1 \\[-2mm]
\multicolumn{4}{c}{\rule{.5pt}{2mm}\rule{4cm}{.5pt}\rule{.5pt}{2mm}}\\[-4mm]
&&\multicolumn{3}{c}{\rule{.5pt}{3mm}\rule{2.8cm}{.5pt}\rule{.5pt}{3mm}}\\[-1mm] \wl4 &\wl1&-&\wl1&-
\end{array}
\)\\[3mm]
\rule{32mm}{0mm}\textbf{and}\qquad\(\displaystyle
\begin{array}[t]{cccc}
  \infty&0&-1&-\frac A{A+B}\\[1mm]
\hline\wl3 &\wl1& \wl1 &\wl1 \\ \wl1 &\wl1&\wl3&\wl1
\end{array}
\)\\

\AAA 33

xyzt \left( x+y \right)  \left( y+z \right)  \left( x-z+t \right) 
 \left( Ax-Ay-Az+Bt \right) 

$1234$, $1235$, $1236$, $1237$, $1238$, $1245$, $1345$, $1456$,
$1457$, $1458$, $1678$, $2345$, $2467$, $2568$, $3468$ 

(1462)(78)

C_2

{(16)(35)}

{\hbox{1356 -- 2478,}\;\;}

{\inf:\nc,\qdd 0:\nc,\qdd 1:\arr3,\qdd 1/2:\arr32,\qdd }

$p^0_4:$  
1347, 1468, 2478, 4567

$p^1_4:$  
1245, 1275, 1285, 2346, 2376, 2386

$p^2_5:$  
12356

$\ l_3:$  
125, 236

\ELLF\qquad\(\displaystyle
\begin{array}[t]{ccccc}
  \infty&0&-1&1&-\frac{A}{A-B}\\[1mm]
\hline\wl2 &\wl2& \wl1 &\wl1&- \\[-1mm] \wl3 &\wl1&\wl1&-&\wl1
\end{array}
\)\\

\AAA 34

xyzt \left( x+y \right)  \left( x+z \right)  \left( x+y+z+t \right) 
 \left( Ay-Az+Bt \right) 

$1234$, $1235$, $1236$, $1237$, $1238$, $1245$, $1345$, $1456$,
$1457$, $1458$, $2345$, $2467$, $2568$, $3468$, $3567$ 

(2546)(78)

D_4

{(25)(78),(26)(35)}

{}

{\inf:\nc,\qdd 0:\nc,\qdd 1:\arr19,\qdd -1:\arr19,\qdd }

$p^0_4:$  
2348, 2467, 3457, 4568

$p^1_4:$  
1245, 1275, 1285, 1346, 1376, 1386

$p^2_5:$  
12356

$\ l_3:$  
125, 136

\ELLF{}

\AAA 35

xyzt \left( Ax+By \right)  \left( Ax+By+At \right)  \left( x+y+z+t
 \right)  \left( By+Bz+At \right) 

$1234$, $1235$, $1236$, $1237$, $1238$, $1245$, $1345$, $1456$,
$1457$, $1458$, $2345$, $2467$, $2568$, $3468$, $3578$ 

(1425)(386)

C_2\oplus C_2

{(14)(26),(12)(38)(46)}

{}

{\inf:\nc,\qdd 0:\nc,\qdd 1:\arr1,\qdd -1:\arr32}

$p^0_4:$  
1368, 1478, 2348, 2367

$p^1_4:$  
1235, 1275, 1285, 3456, 4576, 4586

$p^2_5:$  
12456

$\ l_3:$  
125, 456

\ELLF{}

\AAA 36

xyzt \left( x+y \right)  \left( y-z+t \right)  \left( Ax-By+Bz+At
 \right)  \left( Ax+Ay+Bz+At \right) 

$1234$, $1235$, $1236$, $1237$, $1238$, $1245$, $1345$, $1456$,
$1457$, $1458$, $1678$, $2345$, $2467$, $2568$, $3478$ 

(12)(375)(48)

C_2

{(18)(36)(57)}

{\hbox{1578 -- 2346,}\;\;}

{\inf:\nc,\qdd 0:\nc,\qdd -1:\nc,\qdd -1/2:\arr32,\qdd }

$p^0_4:$  
1368, 1467, 2346, 3458

$p^1_4:$  
1235, 1245, 1265, 2378, 2478, 2678

$p^2_5:$  
12758

$\ l_3:$  
125, 278

\ELLF\qquad\(\displaystyle
\begin{array}[t]{ccccc}
  \infty&0&-1&\frac BA&\frac{A+B}{A}\\[1mm]
\hline\wl2 &\wl1& \wl1 &\wl1&\wl1 \\[-2mm]
&\multicolumn{3}{c}{\rule{.5pt}{2mm}\rule{2.6cm}{.5pt}\rule{.5pt}{2mm}}\\[-4mm]
&&\multicolumn{3}{c}{\rule{.5pt}{3mm}\rule{2.8cm}{.5pt}\rule{.5pt}{3mm}}\\[-1mm] \wl3 &\wl1&\wl1&\wl1&-
\end{array}
\)\\

\AAA 53

xyzt \left( x+y \right)  \left( z+t \right)  \left( Ax-By-Az-At
 \right)  \left( Bx+By-Bz+At \right) 

$1234$, $1235$, $1236$, $1237$, $1238$, $1245$, $1345$, $1467$,
$1468$, $1478$, $1678$, $2345$, $2467$, $2568$, $3467$, $3578$,
$4567$, $4678$ 

(1785426)

C_2\oplus C_2

{(13)(24)(56)(78),(12)(34)}

{\hbox{1257 -- 3468,}\;\;\hbox{1258 -- 3467,}\;\;\hbox{1267 -- 3458,}\;\;}

{\inf:\nc,\qdd 0:\nc,\qdd -1:\nc,\qdd }

$p^0_4:$  
1378, 2478

$p^1_4:$  
1235, 1245, 1285, 1346, 2346, 3476

$p^1_5:$  
12657, 34568

$\ l_3:$  
125, 346

\ELLF\qquad\(\displaystyle
\begin{array}[t]{cccc}
  \infty&0&1&-\frac BA\\[1mm]
\hline\wl4 &-& \wl1 &\wl1 \\[-1mm] -&\wl4 &\wl1&\wl1
\end{array}
\)\\[3mm]
\rule{32mm}{0mm}\textbf{and}\qquad\(\displaystyle
\begin{array}[t]{cccc}
  \infty&0&1&-\frac BA\\[1mm]
\hline\wl3 &\wl1& \wl1 &\wl1 \\ \wl1 &\wl3&\wl1&\wl1
\end{array}
\)\\

\AAA 70

xyzt \left( x-y+z \right)  \left( y-z-t \right)  \left( x-y-t \right) 
 \left( Ax+By \right) 

$1234$, $1235$, $1236$, $1237$, $1238$, $1245$, $1267$, $1345$,
$1367$, $1468$, $1578$, $2345$, $2367$, $2478$, $4567$ 

(12)(368)(57)

C_2

{(37)(45)}

{\hbox{1268 -- 3457,}\;\;}

{\inf:\nc,\qdd 0:\nc,\qdd -1:\arr3,\qdd -1/2:\arr69,\qdd }

$p^0_4:$  
1456, 2346, 2567, 3457

$p^1_4:$  
1268

$p^1_5:$  
12358, 12478

$\ l_3:$  
128

\ELLF\qquad\(\displaystyle
\begin{array}[t]{ccccc}
  \infty&0&1&-1&\frac A{A+B}\\[1mm]
\hline\wl3 &\wl1& \wl1&-&\wl1 \\[-1mm] \wl2 &\wl2&\wl1&\wl1&-
\end{array}
\)\\

\AAA 71

xyzt \left( x+y \right)  \left( x+y+z+t \right)  \left( Ax-By+Az
 \right)  \left( By-Az-At \right) 

$1234$, $1235$, $1236$, $1237$, $1238$, $1245$, $1267$, $1345$,
$1367$, $1468$, $1578$, $2345$, $2367$, $2478$, $3568$ 

(12)(365)(48)

C_2\oplus C_2

{(15)(38)(67),(36)(78)}

{}

{\inf:\arr1,\qdd 0:\nc,\qdd -1:\arr1,\qdd -2:\arr69,\qdd }

$p^0_4:$  
1478, 2348, 2467, 3456

$p^1_4:$  
1245

$p^1_5:$  
12357, 12658

$\ l_3:$  
125

\ELLF{}

\AAA 72

xyzt \left( x+y+z \right)  \left( y+z+t \right)  \left( x-y-t \right) 
 \left( Ay+Bz+Bt \right) 

$1234$, $1235$, $1236$, $1237$, $1238$, $1245$, $1267$, $1345$,
$1367$, $1468$, $1578$, $2345$, $2367$, $2478$, $2568$ 

(1837462)

D_4

{(26)(3745),(26)(34)}

{}

{\inf:\nc,\qdd 0:\arr19,\qdd 2:\arr19,\qdd 1:\nc,\qdd }

$p^0_4:$  
1235, 1247, 1367, 1456

$p^1_4:$  
1268

$p^1_5:$  
23468, 25678

$\ l_3:$  
268

\ELLF{}

\AAA 73

xyzt \left( x+y-z-t \right)  \left( y-z-t \right) \times\\\rule{5cm}{0cm} \times \left( Ax+Ay+Bz+Bt
 \right)  \left( Ax-By+Bt \right) 

$1234$, $1235$, $1236$, $1237$, $1238$, $1245$, $1267$, $1345$,
$1367$, $1468$, $2345$, $2367$, $2478$, $3568$, $4567$  

(1348526)

C_2

{(15)(23)(78)}

{\hbox{1456 -- 2378,}\;\;}

{\inf:\nc,\qdd 0:\nc,\qdd -1:\nc,\qdd -2:\arr69,\qdd }

$p^0_4:$  
1248, 2346, 2378, 3457

$p^1_4:$  
1456

$p^1_5:$  
12567, 13568

$\ l_3:$  
156

\ELLF\qquad\(\displaystyle
\begin{array}[t]{ccccc}
  \infty&0&-1&\frac AB&-\frac B{A+2B}\\[1mm]
\hline\wl1 &\wl1& \wl2 &\wl1&\wl1 \\[-2mm]
&\multicolumn{3}{c}{\rule{.5pt}{2mm}\rule{2.6cm}{.5pt}\rule{.5pt}{2mm}}\\[-4mm]
\multicolumn{5}{c}{\rule{.5pt}{3mm}\rule{5.6cm}{.5pt}\rule{.5pt}{3mm}}\\[-1mm] \wl1 &\wl1&\wl3&\wl1&-
\end{array}
\)\\

\AAA 94

xyzt \left( x+y \right)  \left( x+y+z-t \right)  \left( Ax-By+Az
 \right)  \left( By-Az-Bt \right) 

$1234$, $1235$, $1236$, $1237$, $1238$, $1245$, $1345$, $1467$,
$1568$, $2345$, $2468$, $2578$, $3478$ 

(348675)

1,1

{}

{}

{\inf:\nc,\qdd 0:\nc,\qdd -1:\arr1,\qdd
  $\frac{\sqrt{-3}}2-\frac12$:\arr A,\qdd\\\rule{22mm}{0mm}
  $-\frac{\sqrt{-3}}2-\frac12$:\arr A,\qdd }

$p^0_4:$  
1368, 1478, 2348, 2467, 3456
$p^1_4:$  
1245, 1265, 1285
$p^1_5:$  
12357
$\ l_3:$  
125

\ELLF{}

\AAA 95

xyzt \left( x+y \right)  \left( x+y-z+t \right)  \left( Ax-By+Bz
 \right)  \left( Ax-By-Az-Bt \right) 

$1234$, $1235$, $1236$, $1237$, $1238$, $1245$, $1345$, $1467$, $1568$, $2345$, $2468$, $3678$, $4578$ 

(234875)

1,1

{}

{\hbox{1256 -- 3478,}\;\;\hbox{1357 -- 2468,}\;\;}

{\inf:\nc,\qdd 0:\arr3,\qdd -1:\nc,\qdd -2:\arr93,\qdd }

$p^0_4:$  
1368, 1467, 2468, 3456, 3478

$p^1_4:$  
1245, 1265, 1285

$p^1_5:$  
12357

$\ l_3:$  
125

\ELLF\qquad\(\displaystyle
\begin{array}[t]{ccccc}
  \infty&0&-1&\frac BA&-\frac {A+2B}B\\[1mm]
\hline\wl1 &\wl1& \wl2 &\wl1&\wl1 \\[-2mm]
&\multicolumn{4}{c}{\rule{.5pt}{2mm}\rule{4cm}{.5pt}\rule{.5pt}{2mm}}\\[-4mm]
\multicolumn{4}{c}{\rule{.5pt}{3mm}\rule{4cm}{.5pt}\rule{.5pt}{3mm}}\\[-1mm] \wl3 &\wl1&\wl1&\wl1&-
\end{array}
\)\\

\AAA 96

xyzt \left( x+y \right)  \left( x+y-z+t \right)  \left( Ax-By+Bz+At
 \right)  \left( Ay+Bz+At \right) 

$1234$, $1235$, $1236$, $1237$, $1238$, $1245$, $1345$, $1467$, $1568$, $2345$, $2468$, $2578$, $3678$ 

(12)(368475)

C_2

{(12)(36)(78)}

{\hbox{1278 -- 3456,}\;\;}

{\inf:\nc,\qdd 0:\nc,\qdd -1:\nc,\qdd -1/2:\arr32,\qdd }

$p^0_4:$  
1368, 1467, 2348, 2367, 3456

$p^1_4:$  
1235, 1245, 1265

$p^1_5:$  
12758

$\ l_3:$  
125

\ELLF\qquad\(\displaystyle
\begin{array}[t]{ccccc}
  \infty&0&-1&\frac BA&-\frac A{2A+B}\\[1mm]
\hline\wl1 &\wl1& \wl2 &\wl1&\wl1 \\[-2mm]
&\multicolumn{3}{c}{\rule{.5pt}{2mm}\rule{2.6cm}{.5pt}\rule{.5pt}{2mm}}\\[-4mm]
\multicolumn{5}{c}{\rule{.5pt}{3mm}\rule{5.6cm}{.5pt}\rule{.5pt}{3mm}}\\[-1mm] \wl3 &\wl1&\wl1&\wl1&-
\end{array}
\)\\

\AAA 97

xyzt \left( x+y \right)  \left( x+y+z+t \right)  \left( y-z-t \right) 
 \left( Ax-Bz+At \right) 

$1234$, $1235$, $1236$, $1237$, $1238$, $1245$, $1345$, $1467$, $1568$, $2345$, $2468$, $2567$, $3678$

(1365)(487)

C_2\oplus C_2

{(25)(67),(25)(48),(25)(67)}

{\hbox{1348 -- 2567,}\;\;}

{\inf:\arr19,\qdd 0:\nc,\qdd -1:\nc,\qdd -1/2:\arr93,\qdd }

$p^0_4:$  
1348, 2347, 2368, 3456, 3578

$p^1_4:$  
1235, 1245, 1285

$p^1_5:$  
12657

$\ l_3:$  
125

\ELLF\qquad\(\displaystyle
\begin{array}[t]{ccccc}
  \infty&0&-\frac12&-1&\frac A{B}\\[1mm]
\hline\wl3 &\wl1&- &\wl1&\wl1 \\[1mm] \wl1 &\wl2&\wl1&\wl2&-
\end{array}
\)\\

\AAA 98

xyzt \left( x+y+z \right)  \left( y+z+t \right)  \left( x+z-t \right) 
 \left( Ay+Bz+Bt \right) 

$1234, 1235, 1236, 1237, 1238, 1245, 1345, 1467, 1568, 2345, 2468, 2567, 4578$

(183562)

C_2\oplus C_2

{(17)(26),(26)(34)}

{\hbox{1347 -- 2568,}\;\;}

{\inf:\nc,\qdd 0:\arr19,\qdd 2:\arr93,\qdd 1:\nc,\qdd }

$p^0_4:$  
1235, 1347, 1456, 2457, 3567

$p^1_4:$  
1268, 2568, 2678

$p^1_5:$  
23468

$\ l_3:$  
268

\ELLF\qquad\(\displaystyle
\begin{array}[t]{ccccc}
  \infty&0&1&-1&\frac {A-B}B\\[1mm]
\hline\wl2 &\wl2& \wl1 &\wl1&- \\[1mm] \wl1 &\wl1&-&\wl3&\wl1
\end{array}
\)\\

\AAA 99

xyzt \left( x+y+z \right)  \left( x+z-t \right) \times\\\rule{4cm}{0cm} \times \left( Ax+ \left( A+B
 \right) y-Bz+Bt \right) \left( Ax-By-Bz \right) 

$1234$, $1235$, $1236$, $1237$, $1238$, $1245$, $1345$, $1467$, $1568$, $2345$, $2468$, $2567$, $3478$ 

(1247835)

C_2

{(15)(47)}

{}

{\inf:\nc,\qdd 0:\arr19,\qdd -1:\nc,\qdd -2:\arr19,\qdd }

$p^0_4:$  
1267, 1346, 2456, 2478, 3567

$p^1_4:$  
1458, 1568, 1578

$p^1_5:$  
12358

$\ l_3:$  
158

\ELLF{}

\AAA 100

xyzt \left( x+y-z+t \right)  \left( Ax+Ay+Bz \right)\times\\\rule{6cm}{0cm} \times  \left( Ay+Bz+At
 \right)  \left( By-Bz-At \right) 

$1234$, $1235$, $1236$, $1237$, $1238$, $1245$, $1345$, $1467$, $1568$, $2345$, $2468$, $3567$, $4578$

(167)(2358)

C_2

{(28)(34)(56)}

{\hbox{1278 -- 3456,}\;\;}

{\inf:\nc,\qdd 0:\nc,\qdd -1:\nc,\qdd -2:\arr69,\qdd }

$p^0_4:$  
1236, 1357, 1458, 1467, 3456

$p^1_4:$  
1278, 2578, 2678

$p^1_5:$  
23478

$\ l_3:$  
278

\ELLF\qquad\(\displaystyle
\begin{array}[t]{ccccc}
  \infty&0&-1&\frac BA&-\frac {A+2B}B\\[1mm]
\hline\wl1 &\wl1& \wl2 &\wl1&\wl1 \\[-2mm]
&\multicolumn{4}{c}{\rule{.5pt}{2mm}\rule{4cm}{.5pt}\rule{.5pt}{2mm}}\\[-4mm]
\multicolumn{4}{c}{\rule{.5pt}{3mm}\rule{4cm}{.5pt}\rule{.5pt}{3mm}}\\[-1mm] \wl1 &\wl3&\wl1&\wl1&-
\end{array}
\)\\

\AAA 144

xyzt \left( x-y+z+t \right)  \left( Ax+By+Az \right)\times\\\rule{5cm}{0cm} \times  \left( By+Az+At
 \right)  \left( Bx-By-Az+Bt \right) 

$1234$, $1235$, $1236$, $1237$, $1238$, $1456$, $1457$, $1467$,
$1567$, $2458$, $2678$, $3468$, $3578$, $4567$ 

(175328)

D_4

{(35)(67),(17)(46)}

{\hbox{1467 -- 2358,}\;\;}

{\inf:\nc,\qdd 0:\nc,\qdd -1:\nc,\qdd -2:\arr19,\qdd }

$p^0_4:$  
1236, 1257, 2347, 2456

$p^1_4:$  
1358, 2358, 3458, 3568, 3578

$p^0_5:$  
14678

$\ l_3:$  
358

\ELLF\qquad\(\displaystyle
\begin{array}[t]{ccccc}
  \infty&0&-1&-2&\frac {A}B\\[1mm]
\hline\wl2 &\wl1& \wl2 &\wl1&- \\[1mm] \wl1 &-&\wl1&-&\wl4
\end{array}
\)\\

\AAA 152

xyzt \left( x+y+z+t \right)  \left( y+t \right) \times\\\rule{5cm}{0cm} \times \left( x-y-z+t
 \right)  \left( Ax-Ay+Bz-Bt \right) 

$1234$, $1235$, $1236$, $1237$, $1238$, $1456$, $1478$, $2457$,
$2468$, $3458$, $3567$

(1742)(36)

C_2

{(13)(24)}

{}

{\inf:\arr32,\qdd 0:\arr32,\qdd 1:\arr19,\qdd -1:\nc,\qdd }

$p^0_4:$  
1278, 1356, 1457, 2357, 3478, 5678

$p^1_4:$  
1246, 2346, 2456, 2476, 2486

$\ l_3:$  
246

\ELLF{}

\AAA 153

xyzt \left( x+y+z \right)  \left( y+z+t \right)\times\\\rule{5cm}{0cm} \times   \left( Ax-By+At
 \right) \left( Ax-By+Az+At \right) 

$1234$, $1235$, $1236$, $1237$, $1238$, $1456$, $1478$, $2457$,
$2468$, $3458$, $5678$ 

(1824573)

C_2\oplus C_2

{(14)(56),(14)(38)}

{\hbox{1456 -- 2378,}\;\;}

{\inf:\arr19,\qdd 0:\nc,\qdd -1:\arr3,\qdd -1/2:\arr93,\qdd }

$p^0_4:$  
1235, 1247, 1268, 1456, 2346, 2458

$p^1_4:$  
1378, 2378, 3478, 3578, 3678

$\ l_3:$  
378

\ELLF\qquad\(\displaystyle
\begin{array}[t]{ccccc}
  \infty&0&-1&-2&\frac B{A}\\[1mm]
\hline\wl2 &\wl1& \wl2 &\wl1&- \\[1mm] \wl1 &\wl1&\wl1&-&\wl3
\end{array}
\)\\

\AAA 154

xyzt \left( x+y+z \right)  \left( x+y+z-t \right)\times\\\rule{4cm}{0cm} \times  \left( Ax+ \left( A
+B \right) y-Bz+Bt \right)  \left( Ax-Bz-At \right) 

$1234$, $1235$, $1236$, $1237$, $1238$, $1456$, $1478$, $2457$,
$2568$, $3458$, $3678$ 

(17634285)

C_2

{(17)(38)(45)}

{}

{\inf:\arr1,\qdd 0:\nc,\qdd -1:\nc,\qdd -2:\arr32,\qdd }

$p^0_4:$  
1235, 1267, 1348, 2368, 2478, 3578

$p^1_4:$  
1456, 2456, 3456, 4576, 4586

$\ l_3:$  
456

\ELLF{}

\AAA 155

xyzt \left( Ax+By+Az \right)  \left( Ax+ \left( A+B \right) y-Bz+At
 \right) \times\\\rule{5cm}{0cm} \times \left( Ax-Bz-Bt \right)  \left( Ax+By+Az+At \right) 

$1234, 1235, 1236, 1237, 1238, 1456, 1478, 2457, 2568, 3468, 3578$ 

(174)(2538)

S_3

{(16)(37)(45),(237)(485)}

{}

{\inf:\nc,\qdd 0:\nc,\qdd -1:\nc,\\ \rule{22mm}{0cm}
  $\frac{\sqrt{-3}}2-\frac12$:\arr A,\qdd
  \mbox{$-\frac{\sqrt{-3}}2-\frac12$}:\arr A,\qdd }

$p^0_4:$  
1235, 1278, 1347, 2368, 2467, 3567

$p^1_4:$  
1458, 2458, 3458, 4568, 4578

$\ l_3:$  
458

\ELLF{}

\AAA 197

xyzt \left( x-y-z+t \right)  \left( Ax+By+Bz \right)\times\\\rule{6cm}{0cm} \times  \left( By+Bz+At
 \right)  \left( Ax+Bz+At \right) 

$1234, 1235, 1245, 1267, 1345, 1368, 1478, 2345, 2378, 2468, 5678$

(1287364)

C_2\oplus C_2

{(38)(67),(14)(67),(38)(67)}

{\hbox{1467 -- 2358,}\;\;}

{\inf:\nc,\qdd 0:\nc,\qdd -1:\arr3,\qdd -2:\arr93,\qdd }

$p^0_4:$  
1236, 1278, 1348, 2347, 2358, 2468

$p^0_5:$  
14567

\ELLF\qquad\(\displaystyle
\begin{array}[t]{ccccc}
  \infty&0&1&-\frac BA&-\frac {2B}{A}\\[1mm]
\hline\wl2 &\wl1&-& \wl2 &\wl1 \\[1mm] \wl1 &\wl1&\wl3&\wl1&-
\end{array}
\)\\

\AAA 198

xyzt \left( x+y+z \right)  \left( y+z+t \right)  \left( x-y-t \right) 
 \left( Ax-Ay-Az+Bt \right) 

$1234$, $1235$, $1245$, $1267$, $1345$, $1368$, $1478$, $2345$,
$2378$, $2468$, $3567$

(1285436)

C_2

{(16)(37)}

{}

{\inf:\arr69,\qdd 0:\nc,\qdd -1:\arr19,\qdd -1/2:\arr69,\qdd }

$p^0_4:$  
1235, 1247, 1367, 2346, 2567, 3478

$p^0_5:$  
14568

\ELLF{}

\AAA 199

xyzt \left( x+y+z \right)  \left( y+z+t \right) \times\\\rule{5cm}{0cm} \times \left( Ax+By+ \left( 
A-B \right) z \right)  \left( Ax+By+Az+Bt \right) 

$1234$, $1235$, $1245$, $1267$, $1345$, $1368$, $1478$, $2345$,
$2378$, $2568$, $4567$  

(154732)

C_2

{(15)(23)(46)}

{}

{\inf:\nc,\qdd 0:\nc,\qdd 2:\arr69,\qdd 1:\arr1,\qdd }

$p^0_4:$  
1368, 1456, 2346, 2458, 2678, 3478

$p^0_5:$  
12357

\ELLF{}

\AAA 200

xyzt \left( x+y+z+t \right)  \left( Ax+Ay-Bz-Bt \right)\times\\\rule{6.5cm}{0cm} \times  \left( Ay-Bz+
At \right)  \left( Ax-By-Bt \right) 

$1234$, $1235$, $1245$, $1267$, $1345$, $1368$, $1478$, $2345$, $2568$, $3578$, $4567$

(268)(35)(47)

S_3

{(13)(26)(58),(13)(46)(57)}

{}

{\inf:\nc,\qdd 0:\nc,\qdd -1:\nc,\qdd
  $\frac{\sqrt{-3}}2-\frac12$:\arr A,\qdd\\\rule{22mm}{0cm} 
  \mbox{$\frac{-\sqrt{-3}}2-\frac12$}:\arr A,\qdd }

$p^0_4:$  
1248, 1256, 1467, 2347, 2368, 3456

$p^0_5:$  
13578

\ELLF{}

\AAA 242

xyzt \left( x+y+z \right)  \left( x+z-t \right)\times\\\rule{3cm}{0cm} \times  \left( Ax+ \left( A+B
 \right) y-Bz+Bt \right)  \left(  \left( A+B \right) x+ \left( A+B
 \right) y+Bt \right) 

$1234$, $1256$, $1278$, $1357$, $1368$, $2457$, $2468$, $3456$,
$3478$, $5678$ 

(147685)

G_{64,138}

{(1658)(2437),(1253)(4876),(1736)(2854)}

{\hbox{1235 -- 4678,}\;\;\hbox{1248 -- 3567,}\;\;\hbox{1267 --
    3458,}\;\; \hbox{1346 -- 2578,}\;\; \hbox{1378 -- 2456,}\;\;}

{\inf:\nc,\qdd 0:\arr238,\qdd -1:\nc,\qdd -2:\arr238,\qdd }

$p^0_4:$  
1235, 1248, 1267, 1346, 1378, 2456, 2578, 3458, 3567, 4678

\ELLF\qquad\(\displaystyle
\begin{array}[t]{cccccc}
  \infty&0&1&-1&\frac {A+B}B&-\frac {A+B}{B}\\[1mm]
\hline\wl2 &\wl2& \wl1 &\wl1&-&- \\[1mm] \wl2 &\wl2&-&-&\wl1&\wl1
\end{array}
\)\\[3mm]
\rule{32mm}{0mm}\textbf{and}\qquad\(\displaystyle
\begin{array}[t]{ccccc}
  \infty&0&-1&-\frac {A+B}B&-\frac {A+2B}B\\[1mm]
\hline\wl2 &\wl1& \wl1 &\wl1&\wl1 \\[-2mm]
&&\multicolumn{2}{c}{\rule{.5pt}{2mm}\rule{1.3cm}{.5pt}\rule{.5pt}{2mm}}\\[-4mm]
&\multicolumn{4}{c}{\rule{.5pt}{3mm}\rule{4.2cm}{.5pt}\rule{.5pt}{3mm}}\\[-1mm]
\wl2 &\wl1& \wl1 &\wl1&\wl1 \\[-2mm]
&&\multicolumn{2}{c}{\rule{.5pt}{2mm}\rule{1.3cm}{.5pt}\rule{.5pt}{2mm}}\\[-4mm]
&\multicolumn{4}{c}{\rule{.5pt}{3mm}\rule{4.2cm}{.5pt}\rule{.5pt}{3mm}}\\ 
\end{array}
\)\\

\AAA 243

xyzt \left( x+y+z \right)  \left( y+z+t \right)  \left( x+y+t \right) 
 \left( Ax+By+Az+At \right) 

$1234$, $1256$, $1278$, $1357$, $1368$, $1458$, $2367$, $2457$, $3456$ 

(157642)

S_3

{(34)(57),(13)(67)}

{\hbox{1268 -- 3457,}\;\;\hbox{1367 -- 2458,}\;\;\hbox{1456 -- 2378,}\;\;}

{\inf:\arr239,\qdd 0:\nc,\qdd 1:\arr3,\qdd 1/2:\arr238,\\ \rule{22mm}{0cm} 2/3:\arr240,\qdd }

$p^0_4:$  
1235, 1247, 1268, 1367, 1456, 2346, 2378, 2458, 3457

\ELLF\qquad\(\displaystyle
\begin{array}[t]{cccccc}
  \infty&0&-1&-2&-\frac A{2A-B}&-\frac {B}{A}\\[1mm]
\hline\wl1 &\wl1& \wl2 &-&\wl1&\wl1 \\[-2mm] 
\multicolumn{5}{c}{\rule{.5pt}{2mm}\rule{5.4cm}{.5pt}\rule{.5pt}{2mm}}\\[-4mm]
&\multicolumn{5}{c}{\rule{.5pt}{3mm}\rule{5.6cm}{.5pt}\rule{.5pt}{3mm}}\\[-1mm]
\wl2 &\wl1&\wl2&\wl1&-&-
\end{array}
\)\\

\AAA 244

xyzt \left( x+y+z+t \right)  \left( Ax+Ay+Bz+Bt \right)\times\\\rule{6cm}{0cm} \times  \left( Ay+Bz+
At \right)  \left( Ax+Bz+At \right) 

$1234$, $1256$, $1278$, $1357$, $1368$, $2457$, $2468$, $3458$, $5678$ 

(154)(2863)

C_2\oplus C_2

{(12)(34)(56)(78),(34)(56)}

{\hbox{1278 -- 3456,}\;\;\hbox{1357 -- 2468,}\;\;\hbox{1467 -- 2358,}\;\;}

{\inf:\nc,\qdd 0:\nc,\qdd 2:\arr240,\qdd 1:\nc,\\\rule{22mm}{0cm} 1/2:\arr240,\qdd }

$p^0_4:$  
1256, 1278, 1348, 1357, 1467, 2347, 2358, 2468, 3456

\ELLF\qquad\(\displaystyle
\begin{array}[t]{cccccc}
  \infty&0&-1&\frac{B^{2}}{A(A-2B)}&-\frac {B(2A-B)}{A^{2}}&-\frac {B}{A}\\[1mm]
\hline\wl1 &\wl1& \wl1 &\wl1&-&\wl2 \\[-2mm] 
&\multicolumn{2}{c}{\rule{.5pt}{2mm}\rule{1.3cm}{.5pt}\rule{.5pt}{2mm}}\\[-4mm]
\multicolumn{4}{c}{\rule{.5pt}{3mm}\rule{4.2cm}{.5pt}\rule{.5pt}{3mm}}\\ 
\wl1 &\wl1&\wl1&-&\wl1&\wl2\\[-2mm]
\multicolumn{3}{c}{\rule{.5pt}{2mm}\rule{2.6cm}{.5pt}\rule{.5pt}{2mm}}\\[-4mm]
&\multicolumn{4}{c}{\rule{.5pt}{3mm}\rule{4.8cm}{.5pt}\rule{.5pt}{3mm}\rule{4mm}{0mm}}\\ 
\end{array}
\)\\[5mm]
\rule{32mm}{0mm}\textbf{and}\qquad\(\displaystyle
\begin{array}[t]{cccccc}
  \infty&0&-\frac12&-1&-\frac {A}B&-\frac BA\\[1mm]
\hline\wl1 &\wl2& \wl1 &\wl2&-&- \\[1mm]
\wl1 &\wl1&-& \wl2 &\wl1&\wl1 \\[-2mm]
\multicolumn{2}{c}{\rule{.5pt}{2mm}\rule{1.3cm}{.5pt}\rule{.5pt}{2mm}}
&&&\multicolumn{2}{c}{\rule{.5pt}{2mm}\rule{1.3cm}{.5pt}\rule{.5pt}{2mm}}
\end{array}
\)\\

\AAA 246

xyzt \left( x+y+z \right)  \left( Ax+ \left( A+B \right) y-Bz+Bt
 \right) \times\\\rule{5cm}{0cm} \times \left( Ax-Bz-At \right)  \left( Ax+ \left( A+B \right) y+Az-
At \right) 

$1234$, $1256$, $1278$, $1357$, $1468$, $2358$, $2467$, $3678$, $4578$

(16327)(45)

D_4\oplus C_2

{(23)(46)(78),(14)(37)(56),(14)(28)(56)}

{}

{\inf:\arr1,\qdd 0:\nc,\qdd -1:\nc,\qdd -2:\arr241,\qdd }

$p^0_4:$  
1235, 1268, 1347, 1578, 2378, 2458, 2467, 3468, 3567

\ELLF{}

\AAA 247

xyzt \left( x+y+z \right)  \left( y+z+t \right)  \left( x-y-t \right) 
 \left( Ax-Bz+Bt \right) 

$1234, 1256, 1278, 1357, 1368, 1458, 2367, 5678$

(267584)

D_4

{(26)(34),(37)(45)}

{\hbox{1348 -- 2567,}\;\;\hbox{1578 -- 2346,}\;\;}

{\inf:\nc,\qdd 0:\arr93,\qdd -1:\arr238,\qdd -2:\arr93,\qdd }

$p^0_4:$  
1235, 1247, 1348, 1367, 1456, 1578, 2346, 2567

\ELLF\qquad\(\displaystyle
\begin{array}[t]{cccccc}
  \infty&0&1&-1&-\frac {B}{A+B}&-\frac {A+B}B\\[1mm]
\hline\wl2 &\wl2& \wl1 &\wl1&-&- \\[1mm]
\wl1 &\wl1&\wl1& \wl1 &\wl1&\wl1 \\[-2mm]
\multicolumn{5}{c}{\rule{.5pt}{2mm}\rule{5.4cm}{.5pt}\rule{.5pt}{2mm}}\\[-4.2mm]
&\multicolumn{5}{c}{\rule{.5pt}{3mm}\rule{5.4cm}{.5pt}\rule{.5pt}{3mm}}\\[-4.8mm]
&&\multicolumn{2}{c}{\rule{.5pt}{4.5mm}\rule{1.4cm}{.5pt}\rule{.5pt}{4.5mm}}\\
\end{array}
\)\\

\AAA 248

xyzt \left( x+y+z \right)  \left( y-z-t \right)  \left( x+z+t \right) 
\times\\\rule{7cm}{0cm} \times \left( Ax+ \left( A+B \right) y-Bz+At \right) 

$1234, 1256, 1278, 1357, 1368, 1458, 2358, 2367$

(17)(23)(48)(56)

C_2\oplus C_2

{(15)(23),(15)(46)}

{\hbox{1235 -- 4678,}\;\;}

{\inf:\arr239,\qdd 0:\arr19,\qdd -1:\nc,\qdd -2:\arr245,\qdd\\\rule{22mm}{0cm} -1/2:\arr239,\qdd -2/3:\arr245,\qdd }

$p^0_4:$  
1235, 1267, 1347, 2346, 2378, 2457, 3567, 4678

\ELLF\qquad\(\displaystyle
\begin{array}[t]{cccccc}
  \infty&0&1&-1&\frac {A+B}{B}&\frac {2A+B}{A+B}\\[1mm]
\hline\wl2 &\wl2& \wl1 &\wl1&-&- \\[1mm]
\wl1 &\wl1&\wl2& - &\wl1&\wl1 \\[-2mm]
\multicolumn{5}{c}{\rule{.5pt}{2mm}\rule{5.4cm}{.5pt}\rule{.5pt}{2mm}}\\[-4.2mm]
&\multicolumn{5}{c}{\rule{.5pt}{3mm}\rule{5.4cm}{.5pt}\rule{.5pt}{3mm}}\\
\end{array}
\)\\

\AAA 249

xyzt \left( x+y+z \right)  \left( x+z+t \right) \times\\\rule{4.5cm}{0cm} \times \left( Ax+ \left( A+B
 \right) y-Bz+At \right)  \left( By-Bz+At \right) 

$1234$, $1256$, $1278$, $1357$, $1468$, $2358$, $2467$, $3456$

(183)(57)

C_2\oplus C_2\oplus C_2

{(15)(47)(68),(15)(23),(23)(48)(67)}

{\hbox{1235 -- 4678,}\;\;}

{\inf:\nc,\qdd 0:\nc,\qdd -1:\nc,\qdd -2:\arr241,\qdd }

$p^0_4:$  
1235, 1278, 1346, 2348, 2367, 2456, 3578, 4678

\ELLF\qquad\(\displaystyle
\begin{array}[t]{cccccc}
  \infty&0&1&-1&\frac {B}{A+B}&\frac {A+B}{B}\\[1mm]
\hline\wl2 &\wl2& \wl1 &\wl1&-&- \\[1mm]
\wl1 &\wl1&\wl2& - &\wl1&\wl1 \\[-2mm]
\multicolumn{2}{c}{\rule{.5pt}{2mm}\rule{1.3cm}{.5pt}\rule{.5pt}{2mm}}
&&&\multicolumn{2}{c}{\rule{.5pt}{2mm}\rule{1.3cm}{.5pt}\rule{.5pt}{2mm}}\\
\end{array}
\)\\

\AAA 250

xyzt \left( x+y+z \right)  \left( y+z-t \right)  \left( x+z+t \right) 
 \left( Ax+By-Az+At \right) 

$1234$, $1256$, $1278$, $1357$, $1368$, $2358$, $2467$, $5678$

(173)(458)

C_2

{(16)(45)}

{\hbox{1456 -- 2378,}\;\;}

{\inf:\arr69,\qdd 0:\nc,\qdd 1:\arr245,\qdd -1:\arr93,\qdd\\\rule{22mm}{0cm} -1/2:\arr240,\qdd }

$p^0_4:$  
1235, 1268, 1347, 1456, 2346, 2378, 2457, 3567

\ELLF\qquad\(\displaystyle
\begin{array}[t]{cccccc}
  \infty&0&1&-1&-\frac {B}{A}&-\frac {B}{2A+B}\\[1mm]
\hline\wl2 &\wl2& \wl1 &\wl1&-&- \\[1mm]
\wl1 &\wl1&\wl1& \wl1 &\wl1&\wl1 \\[-2mm]
&\multicolumn{2}{c}{\rule{.5pt}{2mm}\rule{1.3cm}{.5pt}\rule{.5pt}{2mm}}
&\multicolumn{2}{c}{\rule{.5pt}{2mm}\rule{1.3cm}{.5pt}\rule{.5pt}{2mm}}\\[-4mm]
\multicolumn{6}{c}{\rule{.5pt}{3mm}\rule{6.7cm}{.5pt}\rule{.5pt}{3mm}}\\
\end{array}
\)\\

\AAA 251

xyzt \left( x+y+z \right)  \left( x+z-t \right)\times\\\rule{4cm}{0cm} \times  \left( Ax+ \left( A+B
 \right) y-Bz+Bt \right)  \left( Ax-By-Bz-At \right) 

$1234$, $1256$, $1278$, $1357$, $1368$, $2358$, $2467$, $4568$

(1846)

1,1

{}

{}

{\inf:\arr1,\qdd 0:\arr19,\qdd -1:\nc,\qdd -2:\arr93,\\\rule{22mm}{0mm}
  $-\frac{\sqrt5}2-\frac12$:\arr C,\qdd
  \mbox{$\frac{\sqrt5}2-\frac12$}:\arr C,\qdd }

$p^0_4:$  
1235, 1267, 1346, 1458, 2368, 2456, 2478, 3567

\ELLF{}

\AAA 252

xyzt \left( x+y+z \right)  \left( x+y+t \right)\times\\\rule{6cm}{0cm} \times  \left( Ax+2\,Ay-Bz+At
 \right)  \left( Ax-Bz-At \right) 

$1234, 1256, 1278, 1357, 1468, 3456, 3678, 4578$

(152643)(78)

C_2\oplus C_2

{(14)(26),(14)(26),(16)(24)(78)}

{\hbox{1246 -- 3578,}\;\;}

{\inf:\nc,\qdd 0:\nc,\qdd -1:\arr69,\qdd -1/2:\arr241,\qdd }

$p^0_4:$  
1235, 1246, 1348, 1678, 2367, 2478, 3456, 3578

\ELLF\qquad\(\displaystyle
\begin{array}[t]{cccccc}
  \infty&0&-1&-2&\frac {B}{A}&-\frac {2A+B}{A}\\[1mm]
\hline\wl2 &\wl1& \wl2 &\wl1&-&- \\[1mm]
\wl1 &\wl1&\wl1& \wl1 &\wl1&\wl1 \\[-2mm]
\multicolumn{3}{c}{\rule{.5pt}{2mm}\rule{2.6cm}{.5pt}\rule{.5pt}{2mm}}
&\multicolumn{3}{c}{\rule{.5pt}{2mm}\rule{2.8cm}{.5pt}\rule{.5pt}{2mm}}\\[-4mm]
&\multicolumn{4}{c}{\rule{.5pt}{3mm}\rule{4cm}{.5pt}\rule{.5pt}{3mm}}\\
\end{array}
\)\\

\AAA 253

xyzt \left( x+y+z \right)  \left( x+z-t \right)\times\\\rule{4cm}{0cm} \times  \left( Ax+ \left( A+B
 \right) y-Bz+Bt \right)  \left( Ax+Ay-Bz-At \right) 

$1234$, $1256$, $1278$, $1357$, $1368$, $2358$, $2457$, $3467$ 

(18476)(23)

1,1

{}

{\hbox{1267 -- 3458,}\;\;}

{\inf:\arr3,\qdd 0:\nc,\qdd -1:\nc,\qdd -2:\arr240,\qdd\\\rule{22mm}{0mm}
  -1/2:\arr245, \qdd \mbox{$\frac{\sqrt5}2-\frac32$}:\arr C,\qdd
  $-\frac{\sqrt5}2-\frac32$:\arr C,\qdd }

$p^0_4:$  
1235, 1267, 1346, 2368, 2456, 2478, 3458, 3567

\ELLF\qquad\(\displaystyle
\begin{array}[t]{cccccc}
  \infty&0&1&-\frac AB&-\frac {A+B}{B}&\frac {2A+B}{A}\\[1mm]
\hline\wl2 &\wl1& \wl1 &\wl1&\wl1&- \\[-2mm]
&\multicolumn{3}{c}{\rule{.5pt}{2mm}\rule{2.7cm}{.5pt}\rule{.5pt}{2mm}}\\[-4mm]
&&\multicolumn{3}{c}{\rule{.5pt}{3mm}\rule{2.7cm}{.5pt}\rule{.5pt}{3mm}}\\
\wl1 &\wl1&\wl2& \wl1 &-&\wl1 \\[-2mm]
\multicolumn{4}{c}{\rule{.5pt}{2mm}\rule{4cm}{.5pt}\rule{.5pt}{2mm}}\\[-4mm]
&\multicolumn{5}{c}{\rule{.5pt}{3mm}\rule{5.3cm}{.5pt}\rule{.5pt}{3mm}}\\
\end{array}
\)\\

\AAA 254

xyzt \left( x+y+z+t \right)  \left( Ax+Ay-Bz-Bt \right)\times\\\rule{6cm}{0cm} \times  \left( Ay-Bz+
At \right)  \left( Ax-By-Bz \right) 

$1234, 1256, 1278, 1357, 1468, 2358, 3456, 3678$

(265)(384)

C_2

{(14)(23)(56)(78)}

{\hbox{1357 -- 2468,}\;\;}

{\inf:\nc,\qdd 0:\nc,\qdd -1:\nc,\qdd -2:\arr241,\\\rule{22mm}{0mm}
  $\frac{\sqrt5}2-\frac32$:\arr C,\qdd
  \mbox{$-\frac{\sqrt5}2-\frac32$}:\arr C,\qdd }

$p^0_4:$  
1238, 1256, 1357, 1458, 1467, 2347, 2468, 3456

\ELLF\qquad\(\displaystyle
\begin{array}[t]{cccccc}
  \infty&0&-1&\frac BA&\frac {B^{2}}{A(A+2B)}&-\frac {A+2B}{B}\\[1mm]
\hline\wl1 &\wl1& \wl1 &\wl2&\wl1&- \\[-2mm]
&\multicolumn{2}{c}{\rule{.5pt}{2mm}\rule{1.3cm}{.5pt}\rule{.5pt}{2mm}}\\[-4mm]
\multicolumn{5}{c}{\rule{.5pt}{3mm}\rule{5.4cm}{.5pt}\rule{.5pt}{3mm}}\\
\wl1 &\wl1&\wl2& \wl1 &-&\wl1 \\[-2mm]
\multicolumn{4}{c}{\rule{.5pt}{2mm}\rule{4cm}{.5pt}\rule{.5pt}{2mm}}\\[-4mm]
&\multicolumn{5}{c}{\rule{.5pt}{3mm}\rule{5.3cm}{.5pt}\rule{.5pt}{3mm}}\\
\end{array}
\)\\

\AAA 255

xyzt \left( Ax+Ay+Bz+Bt \right)  \left( x+y-2\,z-2\,t \right) \times\\\rule{1cm}{0mm} \times \left( 
Ay+ \left( -2\,A+B \right) z+Bt \right) \left( Bx+ \left( -2\,A+B
 \right) y+ \left( 4\,A-2\,B \right) z-2\,Bt \right)

$1234$, $1256$, $1278$, $1357$, $1368$, $2457$, $3458$, $4678$ 

(15382746)

1,1

{}

{\hbox{1278 -- 3456,}\;\;}

{\inf:\nc,\qdd 0:\nc,\qdd 1/2:\arr32,\\\rule{22mm}{0mm} -1/2:\nc,\qdd
  $\frac{\sqrt5}4+\frac14$:\arr C,\qdd
  $-\frac{\sqrt5}4+\frac14$:\arr C,\qdd }

$p^0_4:$  
1256, 1278, 1357, 1468, 2347, 2368, 3456, 5678

\ELLF\qquad\(\displaystyle
\begin{array}[t]{cccccc}
  \infty&0&\frac12&-\frac AB&-\frac {2A-B}{2B}&\frac {A(2A-B)}{B^{2}}\\[1mm]
\hline\wl2 &\wl1& \wl1 &\wl1&\wl1&- \\[-2mm]
&&\multicolumn{2}{c}{\rule{.5pt}{2mm}\rule{1.3cm}{.5pt}\rule{.5pt}{2mm}}\\[-4mm]
&\multicolumn{4}{c}{\rule{.5pt}{3mm}\rule{4cm}{.5pt}\rule{.5pt}{3mm}}\\
\wl1 &\wl1&\wl1& \wl1 &\wl1&\wl1 \\[-2mm]
\multicolumn{2}{c}{\rule{.5pt}{2mm}\rule{1.3cm}{.5pt}\rule{.5pt}{2mm}}
&&\multicolumn{2}{c}{\rule{.5pt}{2mm}\rule{1.3cm}{.5pt}\rule{.5pt}{2mm}}\\[-4mm]
&&\multicolumn{4}{c}{\rule{.5pt}{3mm}\rule{4.4cm}{.5pt}\rule{.5pt}{3mm}\rule{4mm}{0mm}}\\
\end{array}
\)\\

\AAA 256

xyzt \left( x+y+2\,z \right)  \left( Ay-Bz+Bt \right)\times\\\rule{2cm}{0mm} \times  \left( Ax+Ay+
 \left( 2\,A-B \right) z+Bt \right)   \left( Bx+ \left( -2\,A+B
 \right) y+2\,Bz-2\,Bt \right) 

$1234$, $1256$, $1278$, $1357$, $1368$, $2358$, $2457$, $3456$

(17462)(35)

D_4

{(17)(25)(36)(48),(1847)(2536),(1847)(2536)}

{\hbox{1268 -- 3457,}\;\;\hbox{1367 -- 2458,}\;\;}

{\inf:\nc,\qdd 0:\nc,\qdd 1/2:\arr238,\qdd -1/2:\arr239,\\\rule{22mm}{0mm}
  $\frac{\sqrt{-3}}4+\frac14$:\arr B,\qdd
  $-\frac{\sqrt{-3}}4+\frac14$:\arr B,\qdd }

$p^0_4:$  
1235, 1268, 1367, 2346, 2458, 2567, 3457, 3568

\ELLF\qquad\(\displaystyle
\begin{array}[t]{cccccc}
  \infty&0&1&-\frac {2A}B&-\frac {2A-B}{B}&\frac {2A}{2A-B}\\[1mm]
\hline\wl2 &\wl1& \wl1 &\wl1&\wl1&- \\[-2mm]
&&\multicolumn{2}{c}{\rule{.5pt}{2mm}\rule{1.3cm}{.5pt}\rule{.5pt}{2mm}}\\[-4mm]
&\multicolumn{4}{c}{\rule{.5pt}{3mm}\rule{4cm}{.5pt}\rule{.5pt}{3mm}}\\
\wl1 &\wl2&\wl1& \wl1 &-&\wl1 \\[-2mm]
&&\multicolumn{2}{c}{\rule{.5pt}{2mm}\rule{1.3cm}{.5pt}\rule{.5pt}{2mm}}\\[-4mm]
\multicolumn{6}{c}{\rule{3mm}{0mm}\rule{.5pt}{3mm}\rule{7cm}{.5pt}\rule{.5pt}{3mm}\rule{4mm}{0mm}}\\
\end{array}
\)\\

\AAA 257

xyzt \left( x+y+2\,z+2\,t \right)  \left( Ax+Ay+Bz+Bt \right) \times\\\rule{2cm}{0mm} \times  \left( 
Ay+ \left( -2\,A+B \right) z+Bt \right)  \left(  \left( 2\,A-B
 \right) x-By+ \left( 4\,A-2\,B \right) z-2\,Bt \right) 

$1234$, $1256$, $1278$, $1357$, $1368$, $2457$, $3456$, $4678$

(15)(2748)(36)

C_2

{(23)(68)}

{\hbox{1278 -- 3456,}\;\;\hbox{1367 -- 2458,}\;\;}

{\inf:\nc,\qdd 0:\nc,\qdd 1/2:\nc,\qdd 1/4:\arr240,\\\rule{22mm}{0mm}
  $\frac{\sqrt{-3}}4+\frac14$:\arr B,\qdd
  $\frac{-\sqrt{-3}}4+\frac14$:\arr B,\qdd }

$p^0_4:$  
1256, 1278, 1358, 1367, 2347, 2458, 3456, 5678

\ELLF\qquad\(\displaystyle
\begin{array}[t]{cccccc}
  \infty&0&-1&\frac {2A-B}B&-\frac {2A}{B}&\frac {2A(2A-B)}{B^{2}}\\[1mm]
\hline\wl2 &\wl1& \wl1 &\wl1&\wl1&- \\[-2mm]
&\multicolumn{2}{c}{\rule{.5pt}{2mm}\rule{1.3cm}{.5pt}\rule{.5pt}{2mm}}
&\multicolumn{2}{c}{\rule{.5pt}{2mm}\rule{1.3cm}{.5pt}\rule{.5pt}{2mm}}\\
\wl1 &\wl1&\wl1& \wl1 &\wl1&\wl1 \\[-2mm]
\multicolumn{2}{c}{\rule{.5pt}{2mm}\rule{1.3cm}{.5pt}\rule{.5pt}{2mm}}
&&\multicolumn{2}{c}{\rule{.5pt}{2mm}\rule{1.3cm}{.5pt}\rule{.5pt}{2mm}}\\[-4mm]
&&\multicolumn{4}{c}{\rule{.5pt}{3mm}\rule{4.6cm}{.5pt}\rule{.5pt}{3mm}\rule{4mm}{0mm}}\\
\end{array}
\)\\

\AAA 258

xyzt \left( x-y+2\,z-2\,t \right)  \left( y-z+2\,t \right) \times\\\rule{6cm}{0mm} \times  \left( x-y
+z-t \right)  \left( Ax+By+Az+Bt \right) 

$1234, 1256, 1278, 1357, 1368, 1458, 2367, 4567$

(1367)(245)

C_2

{(15)(36)}

{\hbox{1356 -- 2478,}\;\;}

{\inf:\arr32,\qdd 0:\arr93,\\\rule{22mm}{0mm} -1:\nc,\qdd -2:\arr245,\qdd -1/2:\arr240,\qdd }

$p^0_4:$  
1257, 1356, 1378, 1467, 2346, 2478, 3457, 5678

\ELLF\qquad\(\displaystyle
\begin{array}[t]{cccccc}
  \infty&0&1&\frac {1}2&-\frac {A}{B}&-\frac {A}{2B}\\[1mm]
\hline\wl2 &\wl1& \wl1 &\wl2&-&- \\[1mm]
\wl1 &\wl1&\wl1& \wl1 &\wl1&\wl1 \\[-2mm]
\multicolumn{2}{c}{\rule{.5pt}{2mm}\rule{1.3cm}{.5pt}\rule{.5pt}{2mm}}
&&\multicolumn{2}{c}{\rule{.5pt}{2mm}\rule{1.3cm}{.5pt}\rule{.5pt}{2mm}}\\[-4mm]
&&\multicolumn{4}{c}{\rule{.5pt}{3mm}\rule{4.cm}{.5pt}\rule{.5pt}{3mm}}\\
\end{array}
\)\\

\AAA 259

xyzt \left( x+y+z+t \right)  \left( x-y-z+t \right) \times\\\rule{5cm}{0mm} \times  \left( Ax-Ay+Bz-B
t \right)  \left( Ax-By+Az-Bt \right) 

$1234$, $1256$, $1278$, $1357$, $1368$, $2457$, $3458$, $3467$

(1825)(3746)

C_2\oplus C_2

{(14)(56)(78),(14)(23)}

{\hbox{1456 -- 2378,}\;\;}

{\inf:\arr32,\qdd 0:\arr32,\qdd 1:\nc,\qdd -1:\nc,\qdd }

$p^0_4:$  
1267, 1358, 1456, 2356, 2378, 2458, 3467, 5678

\ELLF\qquad\(\displaystyle
\begin{array}[t]{cccccc}
  \infty&0&1&-1&\frac {A}{B}&-\frac {A}{B}\\[1mm]
\hline\wl2 &-& \wl1 &\wl1&\wl1&\wl1 \\[-2mm]
&&\multicolumn{2}{c}{\rule{.5pt}{2mm}\rule{1.3cm}{.5pt}\rule{.5pt}{2mm}}
&\multicolumn{2}{c}{\rule{.5pt}{2mm}\rule{1.3cm}{.5pt}\rule{.5pt}{2mm}}\\[1mm]
\wl1 &\wl1&\wl1& \wl1 &\wl1&\wl1 \\[-2mm]
\multicolumn{2}{c}{\rule{.5pt}{2mm}\rule{1.3cm}{.5pt}\rule{.5pt}{2mm}}
&\multicolumn{3}{c}{\rule{.5pt}{2mm}\rule{2.7cm}{.5pt}\rule{.5pt}{2mm}}\\[-4mm]
&&&\multicolumn{3}{c}{\rule{.5pt}{3mm}\rule{2.7cm}{.5pt}\rule{.5pt}{3mm}}\\
\end{array}
\)\\

\AAA 261

xyzt \left( x+y+z+t \right)  \left( x-y-z+t \right) \times\\\rule{5cm}{0mm} \times  \left( Ax-Ay+Bz-B
t \right)  \left( Ax+Ay+Bz+Bt \right) 

$1234$, $1256$, $1357$, $1467$, $2358$, $2468$, $3678$, $4578$

(16245)(78)

D_4\oplus C_2

{(14)(23),(56)(78),(1324)(58)(67)}

{\hbox{1258 -- 3467,}\;\;\hbox{1267 -- 3458,}\;\;}

{\inf:\nc,\qdd 0:\nc,\qdd 1:\nc,\qdd -1:\nc,\qdd }

$p^0_4:$  
1258, 1267, 1378, 1456, 2356, 2478, 3458, 3467

\ELLF\qquad\(\displaystyle
\begin{array}[t]{cccccc}
  \infty&0&1&-1&\frac {A}{B}&-\frac {A}{B}\\[1mm]
\hline\wl1 &\wl1& \wl1 &\wl1&\wl1&\wl1 \\[-2mm]
\multicolumn{2}{c}{\rule{.5pt}{2mm}\rule{1.3cm}{.5pt}\rule{.5pt}{2mm}}
&&\multicolumn{2}{c}{\rule{.5pt}{2mm}\rule{1.3cm}{.5pt}\rule{.5pt}{2mm}}\\[-4mm]
&&\multicolumn{4}{c}{\rule{.5pt}{3mm}\rule{4cm}{.5pt}\rule{.5pt}{3mm}}\\
\wl1 &\wl1&\wl1& \wl1 &\wl1&\wl1 \\[-2mm]
\multicolumn{2}{c}{\rule{.5pt}{2mm}\rule{1.3cm}{.5pt}\rule{.5pt}{2mm}}
&\multicolumn{3}{c}{\rule{.5pt}{2mm}\rule{2.7cm}{.5pt}\rule{.5pt}{2mm}}\\[-4mm]
&&&\multicolumn{3}{c}{\rule{.5pt}{3mm}\rule{2.7cm}{.5pt}\rule{.5pt}{3mm}}\\
\end{array}
\)\\

\AAA 262

xyzt \left( x-z-t \right)  \left( Ax+Ay+Bz \right) \times\\\rule{3cm}{0mm} \times  \left( Ax+ \left( 
A+B \right) y-Az+Bt \right)  \left( By+ \left( -A-B \right) z-At
 \right) 

$1234$, $1256$, $1278$, $1357$, $1468$, $2358$, $3467$, $4578$

(163485)(27)

1,1

{}

{}

{\inf:\arr1,\qdd 0:\nc,\qdd -1:\nc,\qdd
  $\frac{\sqrt{-3}}2-\frac12$:\arr A,\\\rule{22mm}{0mm}
  $-\frac{\sqrt{-3}}2-\frac12$:\arr A,\qdd }

$p^0_4:$  
1236, 1345, 1578, 2348, 2457, 2568, 3567, 4678

\ELLF{}

\AAA 264

xyzt \left( y-2\,z+2\,t \right)  \left( Ax+Ay+Bz \right)\times\\\rule{1cm}{0mm} \times   \left( Ax+2
\,Ay+ \left( -2\,A+B \right) z+ \left( 2\,A-B \right) t \right) 
 \left( Ax+Ay-2\,Az+ \left( 2\,A-B \right) t \right) 

$1234$, $1256$, $1278$, $1357$, $2468$, $3456$, $3678$, $4578$

(1384275)

D_4

{(1678)(24)(35),(17)(25)(34),(25)(34)(68)}

{\hbox{1257 -- 3468,}\;\;\hbox{1678 -- 2345,}\;\;}

{\inf:\nc,\ \  0:\nc,\ \  1/2:\nc,\ \ -1/2:\nc,\qdd }

$p^0_4:$  
1236, 1257, 1458, 1678, 2345, 2378, 3468, 4567

\ELLF\qquad\(\displaystyle
\begin{array}[t]{cccccc}
  \infty&0&-\frac12&\frac{A}{B}&-\frac {2A}{2A-B}&-\frac {A}{2A-B}\\[1mm]
\hline\wl1 &\wl1& \wl1 &\wl1&\wl1&\wl1 \\[-2mm] 
\multicolumn{4}{c}{\rule{.5pt}{2mm}\rule{4.2cm}{.5pt}\rule{.5pt}{2mm}}\\[-4mm]
&\multicolumn{4}{c}{\rule{.5pt}{3mm}\rule{4.2cm}{.5pt}\rule{.5pt}{3mm}}\\[-4mm]
&&\multicolumn{4}{c}{\rule{.5pt}{4mm}\rule{4.2cm}{.5pt}\rule{.5pt}{4mm}}\\ 
\wl1 &\wl1&\wl1&\wl1&\wl1&\wl1\\[-2mm]
\multicolumn{3}{c}{\rule{.5pt}{2mm}\rule{2.6cm}{.5pt}\rule{.5pt}{2mm}}
&\multicolumn{3}{c}{\rule{.5pt}{2mm}\rule{3cm}{.5pt}\rule{.5pt}{2mm}}
\\[-4mm]
&\multicolumn{4}{c}{\rule{.5pt}{3mm}\rule{4.cm}{.5pt}\rule{.5pt}{3mm}}\\ 
\end{array}
\)\\[5mm]
\rule{32mm}{0mm}\textbf{and}\qquad\(\displaystyle
\begin{array}[t]{cccccc}
  \infty&0&1&\frac{2A+B}{2A-B}&\frac {2A}{2A-B}&\frac {B^{2}+4A^{2}}{2A(2A-B)}\\[1mm]
\hline\wl2&- &\wl1& \wl1 &\wl2&- \\[1mm]
\wl1 &\wl1&\wl1& \wl1 &\wl1&\wl1 \\[-2mm]
&&\multicolumn{2}{c}{\rule{.5pt}{2mm}\rule{1.3cm}{.5pt}\rule{.5pt}{2mm}}\\[-4mm]
\multicolumn{5}{c}{\rule{.5pt}{3mm}\rule{5.4cm}{.5pt}\rule{.5pt}{3mm}}\\[-4mm]
&\multicolumn{5}{c}{\rule{.5pt}{4mm}\rule{5.4cm}{.5pt}\rule{.5pt}{4mm}}
\end{array}
\)\\

\AAA 265

xyzt \left( x+y-z+2\,t \right)  \left( Ax+2\,Ay-Az+Bt \right)\times\\\rule{4cm}{0mm} \times   \left( 
By-2\,Az+2\,Bt \right)  \left( Bx+By+ \left( 2\,A-B \right) z \right) 

$1234, 1256, 1278, 1357, 1468, 2358, 2467, 3478$

(16587)(34)

C_2\oplus C_2

{(27)(58),(16)(27)(34)}

{\hbox{1267 -- 3458,}\;\;}

{\inf:\nc,\qdd 0:\nc,\qdd 1/2:\nc,\qdd 1/4:\arr69,\qdd }

$p^0_4:$  
1238, 1267, 1357, 2347, 2456, 2578, 3458, 4678

\ELLF\qquad\(\displaystyle
\begin{array}[t]{cccccc}
  \infty&0&-\frac12&\frac{A}{2A-B}&-\frac {2A}{B}&-\frac {A}{B}\\[1mm]
\hline\wl1 &\wl1& \wl1 &\wl1&\wl1&\wl1 \\
[-2mm]
\multicolumn{4}{c}{\rule{.5pt}{2mm}\rule{4cm}{.5pt}\rule{.5pt}{2mm}}\\[-4mm]
&\multicolumn{4}{c}{\rule{.5pt}{3mm}\rule{4cm}{.5pt}\rule{.5pt}{3mm}}\\[-4mm]
&&\multicolumn{4}{c}{\rule{.5pt}{4mm}\rule{4cm}{.5pt}\rule{.5pt}{4mm}}\\[1mm]
\wl1 &\wl2&\wl1& \wl1 &-&\wl1 \\[-2mm]
&&\multicolumn{2}{c}{\rule{.5pt}{2mm}\rule{1.3cm}{.5pt}\rule{.5pt}{2mm}}\\[-4mm]
\multicolumn{6}{c}{\rule{.5pt}{3mm}\rule{6.7cm}{.5pt}\rule{.5pt}{3mm}}\\
\end{array}
\)\\

\AAA 266

xyzt \left( y-2\,z+2\,t \right)  \left( 2\,x+y+2\,t \right) \times\\\rule{4cm}{0mm} \times  \left( Ax
+By+Az \right)  \left( Ax+ \left( A+B \right) y-Az+At \right) 

$1234$, $1256$, $1278$, $1357$, $1368$, $1458$, $2358$, $2467$

(152)(68)

S_3

{(163)(487),(36)(47)}

{}

{\inf:\arr19,\qdd 0:\nc,\qdd 2:\arr245,\qdd -1:\arr245,\\\rule{22mm}{0mm} -2:\arr19,\qdd -4:\arr245,\qdd }

$p^0_4:$  
1237, 1246, 1258, 1356, 2345, 2368, 2567, 4578

\ELLF{}

\AAA 267

xyzt \left( Ax+Ay+ \left( -A+B \right) z \right)  \left( Ax+By-Az+At
 \right) \times\\\rule{4cm}{0mm} \times \left(  \left( -A+B \right) y-Bz+Bt \right)   \left( Bx+By-Az
+Bt \right) 

$1234$, $1256$, $1278$, $1357$, $2468$, $3458$, $3678$, $4567$

(147853)

S_3

{(176)(485),(15)(23)(46)(78)}

{\hbox{1267 -- 3458,}\;\;}

{\inf:\nc,\qdd 0:\nc,\qdd 1:\nc,\\ \rule{22mm}{0mm} $\frac{\sqrt{-3}}2+\frac12$:\nc,\qdd $-\frac{\sqrt{-3}}2+\frac12$:\nc,\qdd }

$p^0_4:$  
1235, 1267, 1468, 1578, 2347, 2368, 3458, 4567

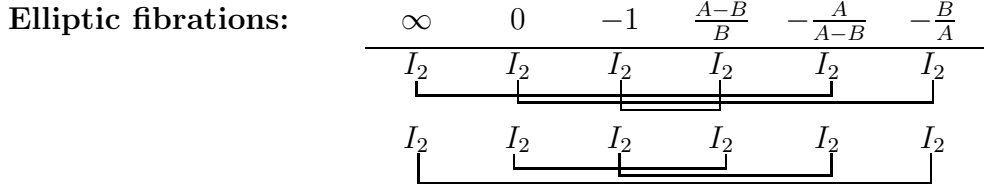
\ELLF\qquad\(\displaystyle
\begin{array}[t]{cccccc}
  \infty&0&-1&\frac{A-B}{B}&-\frac {A}{A-B}&-\frac {B}{A}\\[1mm]
\hline\wl1 &\wl1& \wl1 &\wl1&\wl1&\wl1 \\[-2mm] 
\multicolumn{5}{c}{\rule{.5pt}{2mm}\rule{5.5cm}{.5pt}\rule{.5pt}{2mm}}\\[-4mm]
&\multicolumn{5}{c}{\rule{.5pt}{3mm}\rule{5.5cm}{.5pt}\rule{.5pt}{3mm}}\\[-4.5mm]
&&\multicolumn{2}{c}{\rule{.5pt}{4mm}\rule{1.3cm}{.5pt}\rule{.5pt}{4mm}}\\ 
\wl1 &\wl1&\wl1&\wl1&\wl1&\wl1\\[-2mm]
&\multicolumn{3}{c}{\rule{.5pt}{2mm}\rule{2.8cm}{.5pt}\rule{.5pt}{2mm}}\\[-4mm]
&&\multicolumn{3}{c}{\rule{.5pt}{3mm}\rule{2.8cm}{.5pt}\rule{.5pt}{3mm}}
\\[-4.5mm]
\multicolumn{6}{c}{\rule{.5pt}{4mm}\rule{6.8cm}{.5pt}\rule{.5pt}{4mm}}\\ 
\end{array}
\)\\

\AAA 268

xyzt \left( x+y+z \right)  \left( Ay-2\,Bz+2\,Bt \right)\times\\\rule{3cm}{0mm} \times  \left( 2\,Bx
+2\,By+At \right)  \left(  \left( -A+2\,B \right) x+2\,By-Az+At
 \right) 

$1234, 1256, 1278, 1357, 1468, 2358, 3467, 5678$

(12)(38)(46)(57)

1,1

{}

{\hbox{1268 -- 3457,}\;\;}

$p^0_4:$  
1235, 1247, 1268, 1378, 2346, 2578, 3457, 4568

\ELLF{\inf:\nc,\qdd 0:\nc,\qdd 2:\nc,\qdd -2:\arr69,\\\rule{22mm}{0mm}
  $\sqrt5-1$:\arr C,\qdd $-\sqrt5-1$:\arr C,\qdd }

\qquad\(\displaystyle
\begin{array}[t]{cccccc}
  \infty&0&1&\frac{A-2B}{A}&-\frac {A}{2B}&-\frac {2B}{A}\\[1mm]
\hline\wl1 &\wl1& \wl1 &\wl1&\wl1&\wl1 \\[-2mm] 
\multicolumn{3}{c}{\rule{.5pt}{2mm}\rule{2.8cm}{.5pt}\rule{.5pt}{2mm}}
&\multicolumn{2}{c}{\rule{.5pt}{2mm}\rule{1.3cm}{.5pt}\rule{.5pt}{2mm}}\\[-4mm]
&\multicolumn{5}{c}{\rule{.5pt}{3mm}\rule{5.3cm}{.5pt}\rule{.5pt}{3mm}}\\ 
\wl2 &\wl1&\wl1&\wl1&-&\wl1\\[-2mm]
&\multicolumn{3}{c}{\rule{.5pt}{2mm}\rule{2.8cm}{.5pt}\rule{.5pt}{2mm}}\\[-4mm]
&&\multicolumn{4}{c}{\rule{.5pt}{3mm}\rule{4cm}{.5pt}\rule{.5pt}{3mm}}
\\[-4.5mm]
\end{array}
\)\\

\AAA 270

xyzt \left( x+y+z \right)  \left( y+z+t \right)\times\\\rule{4cm}{0mm} \times  \left( Ax+2\,Ay-Bz+At
 \right)  \left( Bx-2\,Ay+Bz+Bt \right) 

$1234, 1256, 1278, 1357, 1368, 2458, 3467, 5678$

(182)(4567)

C_4

{(1645)(38)}

{\hbox{1456 -- 2378,}\;\;}

{\inf:\nc,\qdd 0:\nc,\qdd -1:\nc,\\\rule{22mm}{0mm} -1/2:\nc,\qdd }

$p^0_4:$  
1235, 1268, 1456, 1478, 2346, 2378, 2458, 3567

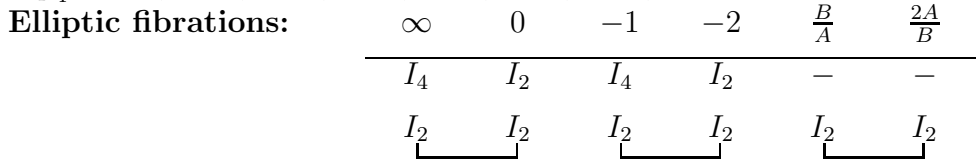
\ELLF\qquad\(\displaystyle
\begin{array}[t]{cccccc}
  \infty&0&-1&-2&\frac {B}{A}&\frac {2A}{B}\\[2mm]
\hline\wl2 &\wl1&\wl2&\wl1&-&-\\[2mm]
\wl1 &\wl1& \wl1 &\wl1&\wl1&\wl1 \\[-2mm] 
\multicolumn{2}{c}{\rule{.5pt}{2mm}\rule{1.3cm}{.5pt}\rule{.5pt}{2mm}}
&\multicolumn{2}{c}{\rule{.5pt}{2mm}\rule{1.3cm}{.5pt}\rule{.5pt}{2mm}}
&\multicolumn{2}{c}{\rule{.5pt}{2mm}\rule{1.3cm}{.5pt}\rule{.5pt}{2mm}}\\ 
\end{array}
\)\\

\AAA 273

xyzt \left( x+y+z \right)  \left( 2\,y+2\,z+t \right)\times\\\rule{6cm}{0mm} \times  \left( 2\,x-2\,
z-t \right)  \left( Ax+2\,By-Az+Bt \right) 

$1234, 1256, 1278, 1357, 1368, 2358, 2367, 4567$

(1258476)

S_3

{(16)(27),(163)(247)}

{}

{\inf:\arr19,\qdd 0:\arr19,\qdd 2:\arr245,\qdd -1:\arr245,\\\rule{22mm}{0mm} -2:\arr19,\qdd -4:\arr245,\qdd }

$p^0_4:$  
1235, 1267, 1347, 1368, 1456, 2346, 2478, 3567

\ELLF{}

\AAA 274

xyzt \left( x+y+z \right)  \left( x+z-t \right)\times\\\rule{3cm}{0mm} \times  \left( Ax+ \left( A+B
 \right) y-Bz+Bt \right)  \left( Ax+Ay-Bz+ \left( A+B \right) t
 \right) 

$1234, 1256, 1278, 1357, 1368, 2358, 2457, 4678$

(13528476)

C_2

{(13)(24)(56)(78)}

{\hbox{1267 -- 3458,}\;\;}

{\inf:\arr3,\qdd 0:\arr3,\qdd -1:\nc,\qdd -2:\arr245,\\\rule{22mm}{0mm}
  -1/2:\arr245, \qdd
  $\frac{\sqrt{-3}}2-\frac12$:\arr B,\qdd
  $-\frac{\sqrt{-3}}2-\frac12$:\arr B,\qdd } 

$p^0_4:$  
1235, 1267, 1346, 1568, 2456, 2478, 3458, 3567

\ELLF\qquad\(\displaystyle
\begin{array}[t]{cccccc}
  \infty&0&1&-\frac{A}{B}&-\frac {A+B}{B}&-\frac {A}{A+B}\\[1mm]
\hline\wl2 &\wl1& \wl1 &\wl1&\wl1&-\\[-2mm] 
&\multicolumn{3}{c}{\rule{.5pt}{2mm}\rule{2.7cm}{.5pt}\rule{.5pt}{2mm}}\\[-4mm]
&&\multicolumn{3}{c}{\rule{.5pt}{3mm}\rule{2.7cm}{.5pt}\rule{.5pt}{3mm}}\\ 
\wl1 &\wl2&\wl1&\wl1&-&\wl1\\[-2mm]
\multicolumn{4}{c}{\rule{.5pt}{2mm}\rule{4cm}{.5pt}\rule{.5pt}{2mm}}\\[-4mm]
&&\multicolumn{4}{c}{\rule{.5pt}{3mm}\rule{4cm}{.5pt}\rule{.5pt}{3mm}}
\\[-4.5mm]
\end{array}
\)\\

\AAA 275

xyzt \left( Bx+ \left( A+B \right) y+ \left( A-B \right) z \right) 
 \left( 2\,Bx+ \left( 2\,A+2\,B \right) z+Bt \right)\times\\\rule{1mm}{0mm} \times \left( 8\,By+
 \left( 4\,A-4\,B \right) z+ \left( A-B \right) t \right)  \left( 
 \left( 2\,A-2\,B \right) x+ \left( -4\,A-4\,B \right) y+ \left( A+B
 \right) t \right) 

$1234$, $1256$, $1357$, $1467$, $2358$, $2478$, $3678$, $4568$

(123584)

C_6

{(173826)(45)}

{\hbox{1235 -- 4678,}\;\;}

{\inf:\nc,\qdd 1:\nc,\qdd -1:\nc,\\\rule{22mm}{0mm} $\sqrt{-3}$:\nc,\qdd $-\sqrt{-3}$:\nc,\qdd }

$p^0_4:$  
1235, 1248, 1346, 1578, 2347, 2567, 3568, 4678

\ELLF\qquad\(\displaystyle
\begin{array}[t]{cccccc}
  \infty&0&-\frac{2B}{A-B}&-\frac {2B}{A+B}&\frac {B^{2}}{(A-B)^{2}}&-\frac{A+B}{A-B}\\[1mm]
\hline\wl1 &\wl1& \wl1 &\wl1&\wl1&\wl1\\[-2mm] 
\multicolumn{4}{c}{\rule{.5pt}{2mm}\rule{4cm}{.5pt}\rule{.5pt}{2mm}}\\[-4mm]
&\multicolumn{4}{c}{\rule{.5pt}{3mm}\rule{4cm}{.5pt}\rule{.5pt}{3mm}}\\[-4mm]
&&\multicolumn{4}{c}{\rule{.5pt}{4mm}\rule{4cm}{.5pt}\rule{.5pt}{4mm}}\\ 
\wl1 &\wl1&\wl1&\wl1&\wl1&\wl1\\[-2mm]
&&\multicolumn{2}{c}{\rule{.5pt}{2mm}\rule{1,3cm}{.5pt}\rule{.5pt}{2mm}}\\[-4mm]
\multicolumn{5}{c}{\rule{.5pt}{3mm}\rule{5.4cm}{.5pt}\rule{.5pt}{3mm}}\\[-4mm]
&\multicolumn{5}{c}{\rule{.5pt}{4mm}\rule{5.4cm}{.5pt}\rule{.5pt}{4mm}}\\
\end{array}
\)\\

\AAA 276

xyzt (x+z+t)(Ax+By+Bt)((-A+B)x-By-Az)(-Ay-Az+(-A+B)t)

$1234$, $1256$, $1357$, $1468$, $2358$, $2478$, $3678$, $4567$

(1738)(265)

C_8

{(15274638)}

{}

{\inf:\nc,\qdd 0:\nc,\qdd 1:\nc,\qdd -1:\nc }

$p^0_4:$  
1237, 1246, 1345, 1678, 2348, 2568, 3567, 4578

\ELLF{}

\AAA D

xyzt (x+y+z+t)\left((\sqrt{-3}+1)x+(\sqrt{-3}+1)y+2z\right)
\left(Ax+(1-\sqrt{-3})By+2Bt\right)  \times\\\rule{5mm}{0mm} \times
\left(((1-\sqrt{-3})A-(1-\sqrt{-3})^2B)x+4Bz+4Bt\right)

$1234$, $1256$, $1278$, $1357$, $1468$, $3458$, $3678$, $4567$

(384675)

S_3

{(37)(46)(58),(35)(47)(68)}

{}

{\inf: \nc, \qdd 0:\arr A, \qdd 2:\arr A, \qdd $1-\sqrt{-3}$:\nc}

$p^0_4:$  
1236, 1247, 1258, 1348, 1567, 3456, 3578, 4678

\ELLF{}

\leavevmode
\subsection{Applications}

Birational transformation between two versions of the surfaces of type
$S_{3}$ described in the section \ref{sec:birat} was used in \cite{CM}
to prove that the double  octic Calabi--Yau threefolds  no. 32 and 69
are birational. Using this birational transformation and a similar
transformation for surface $S_{6}$ together with the elliptic surface
fibration we found thirteen pairs of birational double octics.

Self isogeny of surface of type $S_{2}$ (swapping $I_{2}$ and $I_{4}$
fibers) and quadratic pullbacks of $S_{1}$, $S_{3}$ and $S_{4}$ gives
several examples of correspondences between double octics.

\def\BBB#1 #2 {\mbox{$\left(#1, \ #2\right)$}}
\begin{theorem}
  The following pairs of double octic Calabi--Yau threefolds are
  birational
  \BBB 32 69 , \BBB 10 16 , \BBB 21 53 , \BBB 33 70 , \BBB 36 73 ,
  \BBB 96 100 , \BBB 97  98 , \BBB 153 197 , \BBB 250 258 , \BBB 259
  265 , \BBB 255 268 , \BBB 261 264 , \BBB 267 275 .
\end{theorem}

\begin{theorem}
  There exist correspondences between the following pairs of double
  octic Calabi--Yau threefolds
  \BBB 1 238 , \BBB 32 93 , \BBB 238 241 , \BBB 240 245 , \BBB 2 242 ,
  \BBB 8 249 , \BBB 10 242 , \BBB 247 252 .
\end{theorem}

From the incidence tables used in the classification algorithm we have
derived the groups of permutation of planes that preserves the
incidences. With simple linear algebra one can check which symmetries
correspond to actual automorphisms of the Calabi--Yau threefold. For
any of eleven rigid double octics defined over $\QQ$ the answer is
always yes, as it was already verified in  \cite{BKC} for the
arrangement no. 238 (the one with the largest symmetry group).
In remaining cases the situation is more complicated, invariant
permutation may correspond to an isomorphism with another
member of family (in the case of a family defined over $\QQ$),
isomorphism with the Galois conjugate example (arrangements A, B, C)
or element of the conjugate family (arrangement D). In fact all the
possible phenomenons occurred, since discussion of all examples is
beyond the scope of current paper we shall only give some examples
check if the two Galois--conjugate arrangements of types A, B, C, D
are projectively equivalent and list examples of maximal automorphisms
(cf. \cite{JCR}). 

\begin{theorem}
  The Galois--conjugate arrangements of type A, B or D are
  projectively equivalent, the Galois--conjugate arrangements of type C
  are not projectively equivalent. 
\end{theorem}

\begin{proof}
Permutation $(2,5)(3,7)(4,6)$ corresponds to the
following automorphism 
\[
  \begin{pmatrix}
    x\\y\\z\\t
  \end{pmatrix}
\longmapsto
\begin{pmatrix}
(\sqrt{-3}-1)x\\ (-\sqrt{-3}+1)(x+y)\\
  (-\sqrt{-3}-1)x-2y+(-\sqrt{-3}-1)z\\ (-\sqrt{-3}-1)(x+y+z-t)
\end{pmatrix}
\]
of the projective space $\PP^{3}$ which transforms the arrangement of
type $A$ to its Galois conjugate. 

 Permutation $(13)(24)(56)(78)$ corresponds to the
following automorphism 
\[
(x,y,z,t)\mapsto (z,-t,x,-y)
\]
of the projective space $\PP^{3}$ which transforms the arrangement of
type $A$ to its Galois conjugate. 

Permutation $(3, 5), (4, 7), (68)$
corresponds to the following automorphism 
\[  \begin{pmatrix}
    x\\y\\z\\t    
  \end{pmatrix}
\longmapsto
\begin{pmatrix}
4\,Bx\\ \left( -2\,B\sqrt {-3}+A\sqrt {-3}+2\,B-A \right) y\\ 
\left( 2\,A-4\,B \right) x+ \left( 2\,A-4\,B \right) y+ \left( 2\,A-4\,B \right) z+ \left( 2\,A-4\,B
 \right) t\\
-2\,Ax+ \left( -A\sqrt {-3}-A+2\,B+2\,B\sqrt {-3} \right) y+ \left(
  -2\,A+4\,B \right) t
\end{pmatrix}
\]
of the projective space $\PP^{3}$ which transforms the fiber $(A:B)$ of first family to the fiber
over $(2A:A-2B)$ of the conjugate family.

The only non--trivial symmetry of arrangements of type $C$ is 
$((1, 7), (2, 5), (3, 8), (4, 6))$ and it corresponds to the projective
transformation
\[  \begin{pmatrix}
    x\\y\\z\\t    
  \end{pmatrix}
\longmapsto
\begin{pmatrix}
2\,x+2\,y+ \left( \sqrt {5}-1 \right) t\\ 
\left( -\sqrt {5}+1 \right) x+ \left( -\sqrt {5}+1 \right) y+ \left( -\sqrt {5}+1 \right) z\\ \left( \sqrt {5}-3 \right) x-2
\,y+ \left( \sqrt {5}-1 \right) z+ \left( -\sqrt {5}+1 \right) t\\
\left( -\sqrt {5}+1 \right) y+2\,z-2\,t
\end{pmatrix}
\]
which maps an arrangement of type C onto the same octic arrangement,
\end{proof}

Rohde (\cite{JCR}) used the notion of maximal automorphism of a
Calabi--Yau threefolds to study density of CM--points and give
examples of families of Calabi--Yau threefolds without a point of
Maximal Unipotent Monodromy
\begin{defn}
  A maximal automorphism of a family of Calabi--Yau threefolds is an
  automorphism of a smooth fiber which extends to the local universal
  deformation space.
\end{defn}
\begin{prop}
  One parameter families of double octic Calabi--Yau threefolds
  defined by arrangements No. No. 4, 13, 34, 72, 261 have maximal
  automorphisms acting on $H^{3,0}$ as multiplication by $i$. 
\end{prop}
\begin{proof}
  We give an example of maximal automorphism for each of the above
  arrangements, we  also list the corresponding  permutation of
  planes 

\noindent Arr. No. 4  \(\quad\displaystyle ((1, 5), (3, 6), (7, 8)) \)

\(\displaystyle (x,y,z,t,u)\longmapsto (x+y, -y, y+z, t, iu)\)

\noindent Arr. No. 13 \(\quad\displaystyle  ((1, 3), (5, 6))\)

\(\displaystyle (x,y,z,t,u)\longmapsto (z, y, x, -t, -iu)\)

\noindent Arr. No. 34 \(\quad\displaystyle  ((2, 5), (3, 6))\)

\(\displaystyle (x,y,z,t,u)\longmapsto (x, -x-y, -z-x, -t, iu)\)

\noindent Arr. No. 72 \(\quad\displaystyle  ((3, 4), (5, 7))\)

\(\displaystyle (x,y,z,t,u)\longmapsto (x, -y, -t, -z, -iu)\)
 
\noindent Arr. No. 261 \(\quad\displaystyle  ((1, 3, 2, 4), (5, 8), (6, 7))\)

\(\displaystyle (x,y,z,t,u)\longmapsto (Bz, Bt, Ay, Ax, iA^2B^2u)\)
 \end{proof}

Using \cite[Thm. 7]{JCR} we can deduce the following corollary
\begin{cor}
  One parameter families  of double octic Calabi--Yau threefolds
  defined by arrangements No. No. 4, 13, 34, 72, 261 do not have a
  point of Maximal Unipotent Monodromy.
\end{cor}

The projective automorphism 
\[(x,y,z,t,u)\longmapsto(Bz,By,Bx,At,AB^{3}u)\]
tranforms the equation
\[u^{2}-xyzt(x+y)(y+z)(z+t)(Ax+Bt)\]
of the double octic no. 2 into
\[A^{2}B^{6}(u^{2}-xyzt(x+y)(y+z)(z+t)(Bx+At))\]
the equation of the element corresponding to the parameter $(B:A)$
(multiplied by $A^{2}B^{6}$), and so yields a horizontal
transformation $(A:B)\mapsto (B:A)$ of the family. Consequently this
family is a quadratic pull--back of another family of Calabi--Yau
threefolds.  

We found 54 examples of self transformations of one parameter
families, in the table we collect only the transformation of the
parameter. 

\bigskip

\def\arraystretch{1.2}
\noindent\begin{longtable}{c|l}
\hline Arr. No&\rule{1cm}{0cm}($A,B)\longmapsto \ldots$\\
\hline 2\rule{3mm}{0mm} & \mbox{$(B,A)\ $}\\
\hline 4\rule{3mm}{0mm} & \mbox{$(A,A-B)\ $}\\
\hline 5\rule{3mm}{0mm} & \mbox{$(A,2\,A-B)\ $}\\
\hline 10\rule{3mm}{0mm} & \mbox{$(A,-A-B)\ $}\\
\hline 13\rule{3mm}{0mm} & \mbox{$(A,-A-B),\ $}\mbox{$(B,A),\ $}
                           \mbox{$(B,-A-B),\ $}\mbox{$(A+B,-A),\ $}\mbox{$(A+B,-B)\ $}\\
\hline 21\rule{3mm}{0mm} & \mbox{$(A,-A-B)\ $}\\
\hline 34\rule{3mm}{0mm} & \mbox{$(A,-B),\ $}\mbox{$(B,A),\ $}\mbox{$(B,-A)\ $}\\
\hline 36\rule{3mm}{0mm} & \mbox{$(A+B,-B)\ $}\\
\hline 53\rule{3mm}{0mm} & \mbox{$(B,A)\ $}\\
\hline 71\rule{3mm}{0mm} & \mbox{$(A,-A-B)\ $}\\
\hline 72\rule{3mm}{0mm} & \mbox{$(A,A-B),\ $}\mbox{$(-2\,B+A,-B),\ $}\mbox{$(-2\,B+A,A-B)\ $}\\
\hline 73\rule{3mm}{0mm} & \mbox{$(A,-A-B)\ $}\\
\hline 96\rule{3mm}{0mm} & \mbox{$(A+B,-B)\ $}\\
\hline 97\rule{3mm}{0mm} & \mbox{$(A+B,-B)\ $}\\
\hline 98\rule{3mm}{0mm} & \mbox{$(A,A-B)\ $}\\
\hline 99\rule{3mm}{0mm} & \mbox{$(A,-A-B)\ $}\\
\hline 100\rule{3mm}{0mm} & \mbox{$(A,-A-B)\ $}\\
\hline 144\rule{3mm}{0mm} & \mbox{$(A,-A-B)\ $}\\
\hline 152\rule{3mm}{0mm} & \mbox{$(B,A)\ $}\\
\hline 155\rule{3mm}{0mm} & \mbox{$(A,-A-B)\ $}\\
\hline 198\rule{3mm}{0mm} & \mbox{$(A,-2\,A-B)\ $}\\
\hline 200\rule{3mm}{0mm} & \mbox{$(A,-A-B),\ $}\mbox{$(B,A),\ $}\mbox{$(B,-A-B),\ $}\mbox{$(A+B,-A),\ $}\mbox{$(A+B,-B)\ $}\\
\hline 242\rule{3mm}{0mm} & \mbox{$(A,-A-B),\ $}\mbox{$(A+2\,B,-B),\ $}\mbox{$(A+2\,B,-A-B)\ $}\\
\hline 247\rule{3mm}{0mm} & \mbox{$(A+2\,B,-B)\ $}\\
\hline 248\rule{3mm}{0mm} & \mbox{$(A,-2\,A-B)\ $}\\
\hline 249\rule{3mm}{0mm} & \mbox{$(A,-A-B)\ $}\\
\hline 259\rule{3mm}{0mm} & \mbox{$(B,A)\ $}\\
\hline 261\rule{3mm}{0mm} & \mbox{$(A,-B),\ $}\mbox{$(B,A),\ $}\mbox{$(B,-A)\ $}\\
\hline 264\rule{3mm}{0mm} & \mbox{$(A,-B)\ $}\\
\hline 267\rule{3mm}{0mm} & \mbox{$(A,A-B),\ $}\mbox{$(B,A),\ $}\mbox{$(B,B-A),\ $}\mbox{$(A-B,A),\ $}\mbox{$(A-B,-B)\ $}\\
\hline 270\rule{3mm}{0mm} & \mbox{$(A+B,-2\,A-B)\ $}\\
\hline 273\rule{3mm}{0mm} & \mbox{$(A+2\,B,-B)\ $}\\
\hline 274\rule{3mm}{0mm} & \mbox{$(B,A)\ $}\\
\hline 276\rule{3mm}{0mm} & \mbox{$(B,A)\ $}\\
\hline
\end{longtable}

\bigskip

Similarly, the projective automorphism 
\[(x,y,z,t,u)\longmapsto(A(y+z),-Ay,A(x+y),Bt,A^{3}Bu)\]
tranforms the equation
\[u^{2}-xyzt(x+y)(y+z)(Ax+By+Bz-At)(Ax+Ay+Bz-At)\]
of the double octic no. 2 into
\[A^{6}B^{2}(u^{2}-(-\tfrac BA)xyzt(x+y)(y+z)(Bx+Ay+Az-Bt)(Bx+By+Az-Bt) )\]
the equation of the quadratic twist by $\sqrt{-\frac BA}$ of the element corresponding to the parameter $(B:A)$
(multiplied by $A^{6}B^{2}$). We get only  locally a transformation of
the family overlying the map $(A:B)\mapsto (B:A)$.

We found 32 examples of twisted transformations of one parameter
families, again in the table we collect only the transformation of the
parameter. 

\bigskip

\def\arraystretch{1.2}
\begin{longtable}{c|l}
\hline Arr. No&\rule{1cm}{0cm}($A,B)\longmapsto \ldots$\\
\hline 4 & \mbox{$(B,A)\ $}\mbox{$(B,B-A)\ $}\mbox{$(A-B,A)\ $}\mbox{$(A-B,-B)\ $}\\
\hline 16 & \mbox{$(A+B,-B)\ $}\\
\hline 35 & \mbox{$(A,A-B)\ $}\\
\hline 144 & \mbox{$(A+2\,B,-B)\ $}\mbox{$(A+2\,B,-A-B)\ $}\\
\hline 153 & \mbox{$(A+B,-B)\ $}\\
\hline 155 & \mbox{$(B,A)\ $}\mbox{$(B,-A-B)\ $}\mbox{$(A+B,-A)\ $}\mbox{$(A+B,-B)\ $}\\
\hline 197 & \mbox{$(A,-A-B)\ $}\\
\hline 244 & \mbox{$(B,A)\ $}\\
\hline 246 & \mbox{$(A,-A-B)\ $}\\
\hline 256 & \mbox{$(B,4\,A)\ $}\\
\hline 257 & \mbox{$(2\,A-B,-2\,B)\ $}\\
\hline 264 & \mbox{$(B,-4\,A)\ $}\mbox{$(B,4\,A)\ $}\\
\hline 265 & \mbox{$(A,4\,A-B)\ $}\\
\hline 266 & \mbox{$(2\,A,-A-2\,B)\ $}\mbox{$(4\,B,A)\ $}\mbox{$(4\,B,-A-2\,B)\ $}\mbox{$(A+2\,B,-B)\ $}\mbox{$(2\,A+4\,B,-A)\ $}\\
\hline 273 & \mbox{$(2\,A,-A-2\,B)\ $}\mbox{$(4\,B,A)\ $}\mbox{$(4\,B,-A-2\,B)\ $}\mbox{$(2\,A+4\,B,-A)\ $}\\
\hline 275 & \mbox{$(A-3\,B,A+B)\ $}\mbox{$(A+3\,B,B-A)\ $}\\
\hline
\end{longtable}

\end{document}